\documentclass[11pt,reqno]{amsart}
\usepackage{amsmath,amsfonts,amssymb,amsthm,amsbsy,bbm, amscd,  epsf,calc,enumerate,color,graphicx,verbatim, upref}
\usepackage{ifpdf}
\ifpdf
\pdfoutput=1
    \usepackage[colorlinks,citecolor=blue,linkcolor=blue]{hyperref}
\fi
\usepackage[shortalphabetic,initials]{amsrefs}
\usepackage{color}
\usepackage{datetime}
\numberwithin{equation}{section}
\newcommand{\commentout}[1]{}
\newcommand{\vs}{\vskip.075in}

\newcommand{\R}{\mathbb{R}}
\newcommand{\N}{\mathbb{N}}
\newcommand{\Z}{\mathbb{Z}}

\newcommand{\beq}{\begin{equation}}
\newcommand{\eeq}{\end{equation}}
\newcommand{\ec}{\end{cases}}
\newcommand{\bc}{\begin{cases}}
\newcommand{\ol}{\overline}

\newcommand{\buc}{\text{BUC}}
\newcommand{\rt}{\R^d \times \ [0,T]}

\newcommand {\e}  {\varepsilon}

\newcommand {\df}   {\dfrac}

\newcommand {\ep}  {\epsilon}

\newcommand\round[1]{\left[#1\right]}

\newcommand{\bs}{\begin{split}}
\newcommand{\ess}{\end{split}}
\newcommand{\bea} {\begin{array}{rl}}
\newcommand{\eea} {\end{array}}
\newcommand{\bepa}{\left\{ \begin{array}{l}}
\newcommand{\eepa} {\end{array}\right.}
\newcommand{\too}{\text{o}}
\newcommand{\tO}{\text{O}}
\newtheorem{theorem}{Theorem}[section]
\newtheorem{lemma}[theorem]{Lemma}

\DeclareMathOperator{\DIV}{div}

\begin{document}

\title{Convergence of deterministic growth models}

\author[S. Chatterjee]{Sourav Chatterjee}
\address{Departments of Mathematics and Statistics\\ Stanford University}
\email{souravc@stanford.edu}
\author[P.~E. Souganidis]{Panagiotis E. Souganidis}
\address{Department of Mathematics\\ The University of Chicago}
\email{souganidis@uchicago.edu}

\begin{abstract}
We prove the uniform in space and time convergence of the scaled heights of large classes of deterministic growth models that are monotone and equivariant under translations by constants. The limits are characterized as the unique (viscosity solutions) of first- or second-order  partial differential equations depending on whether the growth models are scaled hyperbolically or parabolically. The results greatly simplify and extend a recent work by the first author to more general surface growth models. The proofs are based on the methodology developed by  Barles and the second author to prove convergence of approximation schemes.
\end{abstract}

\maketitle

 \section{Introduction}\label{intro}
 
In this note we prove the uniform in space and time convergence of the scaled heights of large classes of deterministic growth models that are monotone and equivariant under translations by constants. The limits are unique viscosity solutions  of  first- or second-order  partial differential equations (PDEs for short) depending on whether the growth models are scaled hyperbolically or parabolically. Examples of such equations in the parabolic scaling include the deterministic Kardar--Parisi--Zhang (KPZ) equation as well as nonlinear PDE  with discontinuities in the gradient that are ``compatible''
with  Finsler norms, like, for example, the crystalline infinity Laplacian, that is,   
the infinity Laplacian 
corresponding to the $l^1$-metric in $\R^d$.

\vs

Our results are based on nonlinear PDE techniques (viscosity solutions) and, in particular, the methodology developed  by Barles and the second author  (Barles and Souganidis~\cite{BS}) to prove convergence of monotone approximation schemes.
\vs

Since regularity plays no role, we are able to study very general and broad settings of deterministic growth models with nonsmooth generating (height) functions. This leads in the limit  to some unusual PDE with discontinuities.  Moreover, the scaling need not always be parabolic.
\vs

Our method greatly simplifies  a recent work of the first author (Chatterjee~\cite{C}; see also \cite{C3}), whose main focus was convergence of smooth height functions to the deterministic Kardar--Parisi--Zhang (KPZ) equation. The results of \cite{C}, which  were based on linear PDE estimates requiring higher ($C^2$) regularity, imposed  more assumptions.

\vs 

The investigation is motivated by activity surrounding the KPZ equation in probability theory and mathematical physics, although the investigation in \cite{C} and in this paper are about deterministic, rather than random, surface growth. The KPZ equation is a stochastic PDE, introduced by Kardar, Parisi and Zhang in \cite{KPZ86}, which  is conjectured to be the  ``universal scaling limit'' of a large class of growing random surfaces. In spite of tremendous progress in the last twenty years (for a very brief survey, see \cite[Section 1.4]{C}), this conjecture remains largely open in dimensions higher than one. In \cite{C}, it was shown that if the randomness is dropped,  then a general class of parabolically scaled deterministically growing   surfaces   converge to the solution of a deterministic version of the KPZ equation. 
 
\vs 
As an example of the type of results obtained in this paper, we describe next  a version of the zero temperature Glauber dynamics of a gradient Gibbs measure with potential $V$ (see Subsection \ref{gibbs}).

\vs 
Suppose that we have a $d$-dimensional surface growing deterministically according to the following rule: Let $f(x,t)$ denote the height of the surface at a point $x\in \Z^d$ at time $t\in \Z_+$. Then $f(x,t+1)$ is the middle point of the values of $y$ that minimize
\[
\sum_{i=1}^d V(y - f(x+e_i,t)) + \sum_{i=1}^d V(y - f(x-e_i, t)),
\]
where $e_1,\ldots,e_d$ are the standard basis vectors of $\R^d$ and $V$ is a convex symmetric potential function. 
\vs 

We obtain scaling limits of such surfaces for a large class of potentials, both smooth and non-smooth. For example, for the smooth potential $V(x)=x^4$, the parabolic scaling limit yields the PDE 
\[u_t=\dfrac{\sum_{i=1}^d (Du,e_i)^2 (D^2u \; e_i, e_i)}{2|Du|^2},\]
where $u_t$ is the partial derivative of $u$ with respect to time, $Du$ and $D^2u$  are respectively the gradient  and Hessian of $u$ with respect  to $x$, and $(\cdot,\cdot)$ is the inner product in $\R^d$. 
\vs

The PDE of the scaling limits for non-smooth  $V$  are more complicated. For instance, when $d=2$ and the potential is $V(x)=|x|$ (coming from the solid-on-solid model of statistical physics), the parabolically scaled limit is the viscosity solution (see Section \ref{schemesec} for the definition of the solution and the precise meaning of the PDE) of  
\[u_t=F(D^2u, Du) \ \ \text{in} \ \ \R^2\times (0,T],\]
where 
\[F(D^2u,Du)=\bc \frac{1}{2} u_{x_1x_1} \ \ \text{if} \  \ |u_{x_1}| > |u_{x_2}|, \\[2mm]
 \frac{1}{2} u_{x_2x_2} \ \ \text{if} \ \ |u_{x_2}| > |u_{x_1}|.  \ec\]
This provides a vast generalization of the setting of \cite{C}, where only examples leading to a deterministic KPZ scaling limit were considered. 
 
\vs
The paper is organized as follows. In the next section we discuss the general setting and describe the growth models we consider. In Section~\ref{schemesec} we present a variant of the argument of \cite{BS} which we use to prove our results in the next two sections. Section~\ref{hyperbolicsec} is about the hyperbolic scaling. The last section of the paper is devoted to the parabolic scaling. Since the results depend on the regularity and the nature of the minimum of the convex potential $V$, Section~\ref{parabolicsec} is divided into several subsections.

 \section{The general setup}\label{setup}
 
 We describe here the general scheme that will give convergence of the scaled heights of deterministic surface growth models. The goal is to formulate the  algorithm yielding the convergence, as described in \cite{C}, in a way that will allow us to use the methodology of ~\cite{BS} to prove convergence of approximation schemes.  The latter, which is described in the next section, yields the convergence of the scaled height functions to the unique (viscosity) solution of PDE associated with the specific growth model and scaling. 
 \vs
 
 The approach put forward here provides a considerably simpler proof of the main result of \cite{C} and, more importantly,
 allows the study of more general deterministic surface growth models with non-smooth generating functions which give rise to ``unusual'' first- and second-order partial differential equations. 
 
 \vs
 
Our presentation of the scheme is based on the general setting of \cite{C}. A $d$-dimensional discrete surface is a function from $\Z^d$ into $\R$, whose value at a point denotes the height of the surface at that point. We consider discrete surfaces evolving  over time according to some deterministic local rule, to be made precise below. 
\vs

Let $\{e_1, \ldots, e_d\}$  be the standard basis vectors of $\R^d$, $\Z_+$ the set of nonnegative integers, $\R_+=(0,\infty)$ and $\ol \R_+=[0,\infty)$. We denote by $A$ the set $\{0,\pm e_1, \ldots, \pm e_d\}$ consisting of the origin and its $2d$ nearest neighbors in $\Z^d$. Let $B = A \setminus \{0\}$. A surface growth model is described by some $\varphi: \R^A \to \R$, which is assumed to be equivariant under constant shifts and  monotone,  properties which are explained later in the paper.

\vs

We say that the evolution of a deterministically growing $d$-dimensional surface 
$u:  \Z^d \times \Z_+ \to \R$  is driven by $\varphi$ if, for each $(x,t) \in  \Z^d \times \Z_+$,  
\beq\label{takis1100}
u(x, t + 1) = \varphi((u(x + a,t))_{a \in A}).
\eeq

Throughout the discussion, we will be assuming that 
\[
\varphi (0,\ldots,0)=0.
\]
This causes no loss of generality, since the property of equivariance under constant shifts (explained later) will ensure that, if we replace $\varphi$ by 
\[
\widetilde \varphi =\varphi - \varphi (0,\ldots,0),
\]
then the new surface $v(x,t)$ is related to the old surface $u$ by  $v(x,t) = u(x,t)-t\varphi(0,\ldots,0)$. Henceforth, we will write $\varphi(0)$ instead of $\varphi(0,\ldots,0)$, for ease of notation. 
\vs

We are interested in the long-time and large-space behavior of the height function $u$. For this, it is convenient to extend $u$ to a function on $\R^d \times [0,\infty)$ and to scale space-time either  hyperbolically or parabolically. The choice of the scale depends on properties of the generating function $\varphi$.
\vs

To describe the scalings and the extension, we recall that, given $t\in \R$, $\round t$ denotes its integer part, and, for $x\in \R^d$, $ \round x=(\round{x_1},\ldots, \round{x_d})$.  
\vs
We start with the hyperbolic scaling. Given $\ep>0$, we assume that, for some given  $u_0:\R \to \R$, 
\[
u(x,0)=u_0(\ep x), 
\]
and generate $u(x,t)$ for $t>0$ by \eqref{takis1100}. Then, we define $u^\ep:\R^d\times [0,\infty)\to \R$  by 
\[
u^\ep(x,t)=u(\round{ \frac{x}{\ep}}, \round{\frac {t}{\ep}}).
\]
It is immediate that 
\[
u^\ep(x,t)= u_0(\ep \round{\frac{x}{\ep}}) \ \ \text{if} \ \ t\in [0,\ep).
\]
\vs
Next, we examine how $u^\ep$ evolves in time, in view of \eqref{takis1100}. We  use the elementary fact that $\round {t+1}=\round t +1$ and \eqref{takis1100} to get, for $t\geq \ep$,  the following string of equalities:
\beq\label{takis1104.1}
\begin{split}
u^\ep(x,t)&= u(\round{ \frac{x}{\ep}}, \round{\frac {t}{\ep}})=\varphi\left( \left (u(\round{ \frac{x}{\ep}}+a, \round{\frac {t}{\ep}}-1)\right )_{a\in A}\right)\\[1.2mm]
&= \varphi \left(\left (u(\round{ \frac{x +\ep a}{\ep}}, \round{\frac {t -\ep}{\ep}})\right )_{a\in A}\right).
\end{split}
\eeq
Hence, we have 
\[
u^\ep(x,t)=S(\ep) u^\ep(\cdot, t-\ep)(x),
\]
where, given $v:\R^d\to\R$,  
\[
S(\ep) v(x)= \varphi ((v(x + \ep a))_{a \in A}).
\]
For the parabolic scaling, given $\ep>0$, we assume that 
\[
u(x,0)=u_0(\sqrt{\ep} x), 
\]
for some given  $u_0:\R \to \R$, generate $u$ according to \eqref{takis1100}, 
and define $u^\ep:\R^d\times [0,\infty)\to \R$  by 
\[
u^\ep(x,t)=u(\round{ \frac{x}{\sqrt{\ep}}}, \round{\frac {t}{\ep}}).
\]
It is immediate that 
\[
u^\ep(x,t)= u_0(\sqrt{\ep} \round{\frac{x}{\sqrt{\ep}}}) \ \text{if} \ t\in [0,\ep).
\]
and, as above, we find that 
\[
u^\ep(x,t)=S(\ep) u^\ep(\cdot, t-\ep)(x),
\]
where, given $v:\R^d\to\R$,  
\[
S(\ep) v(x)= \varphi ((v(x + \sqrt{\ep} a))_{a \in A}).
\]
We note that we used the same notation for both the hyperbolic and parabolic scalings. We expect this to create no difficulties in what follows since  the arguments and statements will always specify which case we work with. The benefit, however, is that we do not need to introduce more notation.  
\vs
Finally, we remark that, for all $g\in \buc(\R^d)$, as $\delta\to 0$ and  uniformly in $x$,
\[g(\delta \round{\frac{x}{\delta}}) \to g(x),\]
where $\buc(\mathcal O)$ is the space of real-valued bounded uniformly continuous functions on $\mathcal O \subset \R^m$ for some $m\in \N$. In this paper, $\mathcal O$ is either $\R^d$ or $\R^d\times [0,T]$ for some $T>0$. 

\vs

The basic mathematical criterion about which scaling to use is how the scheme propagates linear functions. If such functions move in time, then the correct scaling is the hyperbolic one. If, however, linear functions remain the same, to see some nontrivial behavior we need to scale parabolically. 
\vs

The intuition behind the choice of scaling can also be described as follows. The scaling limit describes the long time and large space behavior of the growth process. The space scaling keeps the problem in a ``compact'' set in space while the time scaling can be thought heuristically as an expansion in $1/t, 1/t^2, \ldots$. The hyperbolic scaling  gives the $1/t$ term. If this is $0$, which is the case when linear functions do not move, then one goes to the $1/t^2$ term, hence the parabolic rescaling.
\vs

Finally,  from the modeling point of view, hyperbolic scaling may occur in any growth mechanism where the growth at a point is affected by the heights at only those neighboring points where the height is higher.

 \section{The approximation scheme}\label{schemesec}

 We describe here a reformulation of the abstract method put forward in \cite{BS} to establish  the (local uniform) convergence of approximations to the (viscosity) solution of the initial value problem
 \beq\label{ivp1}
 u_t=F(D^2u,Du) \ \ \text{in} \ \ \R^d\times (0,T] \qquad 
 u(\cdot,0)=u_0 \ \ \text{in} \  \ \R^d,
 \eeq
 with  $F:{\mathcal S}^d \times \R^d \to \R$ degenerate elliptic, 
 that is,
 \[
 F=F(X,p)  \text{ is increasing with respect to $X\in {\mathcal S}^d$},
 \]
where ${\mathcal S}^d$ is the space of symmetric $d\times d$ matrices and monotonicity is  interpreted in the sense of symmetric matrices, that is, $X\ge Y$ if $X-Y$ is positive semidefinite.  The scheme, presented below, asserts that monotone and  translation equivariant schemes that are consistent with \eqref{ivp1} converge 
(locally uniformly) to the unique Crandall-Lions viscosity solution   $u\in \text{BUC}(\R^d\times [0,T])$ of \eqref{ivp1}. For the convenience of the reader we recall the definition further down in this section. Notice that heretofore when we refer to sub-, super- and solutions, we always mean in the viscosity sense, that is, they are viscosity sub-, super- and solutions.
  \vs
Note that, in view of the assumed degenerate ellipticity of $F$, the method also works for  first-order Hamilton-Jacobi initial value problems like
 \[
  u_t=H(Du) \ \ \text{in} \ \ \R^d\times (0,T] \qquad
 u(\cdot,0)=u_0 \ \ \text{in} \  \ \R^d.
\]
In many of the  examples of surface growth models we study in this paper, the parabolically rescaled limits give rise to equations like \eqref{ivp1} with nonlinearities $F=F(X,p)$ which have discontinuities in the gradient component $p$. In such cases, it is more convenient and actually necessary to reinterpret to  and relax  \eqref{ivp1} as  two inequalities corresponding to sub- and super-solutions, that is, to consider the ``relaxed'' initial value problem 
\beq\label{general_ivp1}
u_t \leq \overline{F}(D^2u,Du)  \ \ \text{in} \ \  \R^d\times (0,T)  \ \ u_t \geq \underline{F}(D^2u, Du) \ \ \text{in} \ \  \R^d\times (0,T) \ \ u(\cdot,0)=u_0,
\eeq
where $\overline F \in \text{USC}({\mathcal S}^d\times \R^d) $ and $\underline F \in \text{LSC}({\mathcal S}^d\times \R^d)$. Here  $\text{USC}(U)$ and $\text{LSC}(U)
$ are respectively the sets of upper and lower semicontinuous functions on $U$.

 \subsection{The general setting and assumptions}  
 We work in $\mathcal B$,  the set of bounded functions $u:\R^d\to \R^d$, that is, functions satisfying  
 \[
 \|u\| =\underset{x\in \R^d}\sup |u(x)| <\infty.
 \]
For $\ep \in [0,1]$, let $S(\ep):\mathcal B \to \mathcal B$ be such that, for all $u,v\in \mathcal B, k\in \R$ and $\ep \in [0,1]$,
\beq\label{id}
S(0)u=u,
\eeq  
\beq\label{translation}
S(\ep)(u+k)=S(\ep)u+k,
\eeq
and
\beq\label{monotonicity}
\text{if} \  \ u\leq v, \ \text{ then} \  S(\ep)u\leq S(\ep)v.
\eeq  

The last two conditions  are referred to respectively as equivariance (or invariance) under translation by constants and monotonicity. 
A well-known observation of Crandall and Tartar \cite{CT} gives that, if~\eqref{translation} holds, then monotonicity is equivalent to contraction, that is, for all   $u,v\in \mathcal B$ and $\ep \in [0,1]$, 
\beq\label{contraction}
\|S(\ep)u-S(\ep)v\|\leq \|u-v\|.
\eeq
Next, we discuss the main assumption about $S(\ep)$, which connects it with  \eqref{ivp1}. Since we are aiming for some generality in order to incorporate all the examples we have in mind, the following assumption may appear a bit cumbersome.  
\vs
We assume that the family of operators $(S(\ep))_{\ep \in [0,1]}$ is such that 
\beq\label{consistency} 
\bc
\text{there exist degenerate elliptic $\overline F\in \text{USC}(\mathcal{S}^d\times \R^d)$ and $\underline F\in \text{LSC}(\mathcal{S}^d\times \R^d)$ such that,}\\[1.5mm]
 \text{for all $\phi \in C^2(\R^d)$ and $x\in \R^d$,} \\[2mm]
\underset{y\to x, \ep\to 0}\limsup \dfrac{S(\ep) \phi (y) -\phi (y)}{\ep} \leq \overline {F}(D^2\phi(x), D\phi(x)) \\[3mm]
\text{and}\\[1.5mm]
\underset{y\to x, \ep\to 0 }\liminf \dfrac{S(\ep) \phi (y) -\phi (y)}{\ep} \geq  \underline{F}(D^2\phi(x), D\phi(x)). 
\ec
\eeq
Next, we assume that  the initial value problem \eqref{general_ivp1}
satisfies a comparison principle between bounded upper semicontinuous (\text{BUSC} for sort) subsolutions and bounded lower semicontinuous subsolutions (\text{BLSC} for short)  supersolutions, that is, 
\beq\label{lh1}
\bc
\text{if $w\in \text{BUSC}(\R^d\times [0,T])$ and $v\in \text{BLSC}(\R^d\times [0,T])$ satisfy}\\[1.5mm]
w_t\leq \overline F(D^2w,Dw) \ \ \text{and} \ \
v_t \geq \underline{F}(D^2v, Dv) \ \ \text{in} \ \  \R^d\times (0,T), \ \ 
\text{and} \ \ w(\cdot,0)\leq v(\cdot,0),\\[1.5mm] 
\text{then} \ \ w\leq v\ \ \text{in} \ \ \R^d \times [0,T].
\ec
\eeq
As already noted we work with the Crandall--Lions viscosity solutions of \eqref{ivp1} with $F$ degenerate elliptic, and  refer to the Crandall, Ishii and Lions  ``User's Guide'' \cite{CIL} for an extensive introduction to the theory.  For the reader's convenience, we recall here the definition of subsolution (resp.~subsolution) of \eqref{ivp1}. 
\vs

We say that  $u\in \text{USC}(\R^d\times (0,T])$ (resp.~$u\in \text{LSC}(\R^d\times (0,T]))$)  is a subsolution (resp.~supersolution)  of $u_t\leq \overline F (D^2u,Du)$ (resp. $u_t\geq \underline F (D^2u,Du)$) if,  for  every $\phi \in C^2(\R^d)$ and $g \in C^1((0,T])$ and a maximum  (resp.~minimum) point $(x_0,t_0) \in \R^d\times (0,T]$ of $u-\phi-g$,
\beq\label{def}
g'(t_0)\leq \overline {F} (D^2\phi(x_0), D\phi(x_0)) \ \ (\text{resp.} \ \  g'(t_0)\geq \underline {F}(D^2\phi(x_0), D\phi(x_0)) ).
\eeq

A function that is both a subsolution and a supersolution is called a solution. 
\vs
We remark here, and refer to \cite{CIL} for more discussion and proofs,  that in the definition of a subsolution (resp.~supersolution) maxima (resp.~minima) can be either local or global, and, finally, they can always be taken to be strict. 
\vs
We also note and refer to \cite{CIL} for more discussion that there is a great freedom in choosing  the regularity of the test function in  \eqref{def}, the general principle being that $\phi$ must be sufficiently regular so that \eqref{def} makes sense. As a consequence, if $\overline F$ and $\underline F$ do not depend on the Hessian,  it suffices to use test functions in $C^1(\R^d)$.
\vs
Finally, as it will become clear later, \eqref{consistency}  is used to check that a certain function is a subsolution (resp.~supersolution). Hence, the regularity of $\phi$ in \eqref{consistency} needs to be the same as the one of the test function used in \eqref{def}.
\vs

A few remarks are in order to explain the relationship between \eqref{ivp1} and \eqref{lh1}. In all the examples we investigate in this note, either $F=\overline F=\underline F$  in $\mathcal {S}^d \times \R^d$ or in $\mathcal{S}^d\times (\R^d\setminus \ol U)$,  where $U$ is a  subset of $\R^d$. In the latter case, on $\mathcal{S}^d \times U$ we have 
\begin{align}\label{fstar1}
\ol {F}(X,p)= F^\star(X,p):=\underset{\mathcal{S}^d \times  \R^d \ni (Y,q)\to (X,p)} \limsup F(Y,q)
\end{align}
and
\begin{align}\label{fstar2}
\underline {F}(X,p)=F_\star(X,p):=\underset{\mathcal{S}^d \times  \R^d \ni (Y,q)\to (X,p)} \liminf F(Y,q).
\end{align}

\vs

When $F$ is continuous, that is, $U=\emptyset$ in \eqref{consistency}, the comparison principle is a classical fact in the theory of viscosity solutions; see, for example, see Theorem~8.3 in \cite{CIL}.
\vs

When $F$ has discontinuities, the comparison principle, if true, depends very much on the type of singularities. 
The folklore of the theory of viscosity solutions is that discontinuities can be dealt with by identifying and using an appropriate class of test functions which are consistent with the classical theory and ``resolve the discontinuities''. The latter means that $\overline F=\underline F$ when evaluated along the Hessians and gradients of the new test functions. 
\vs

When $F$ is discontinuous at $p=0$, as in, for example, \eqref{ex2.3.aa}, the comparison principle follows from the techniques developed by Chen, Giga and Goto \cite{CGG}, Evans and Spruck \cite{ES} and Ishii and Souganidis \cite{IS}.
When the singularities are at $p=0, p_1, \ldots, p_k$, the  comparison principle follows as in  Gurtin, Soner and Souganidis \cite{GSS}, Ohnuma and Sato \cite{OS} and Ishii \cite{I}. The last reference treats  some $F$'s with singularities of the type arising in this paper with the restriction that  the set of discontinuities is smooth,  which is not the case in dimensions higher than $2$.
\vs 

When $d\geq 3$, the typical $U$ arising in this paper  does not have smooth boundary and new arguments are needed. The necessary comparison in this generality was established recently by Morfe and the second author \cite{MS}. The key observation  of \cite{MS} is that gradient discontinuities  are  consistent with particular polyhedral Finsler norms in $\R^d$, like, for example, $l^1$. This allows to construct the correct test functions. Specific comments are made when necessary in the paper. More discussion, however,  on this subject is beyond the scope of the paper at hand.
\vs

\subsection{Convergence of the approximation scheme} 
Let $(S(\ep))_{\ep\in[0,1]}$ be a family of maps as in the previous subsection. 
Fix $T>0$. Given $u_0\in \mathcal B$ and $\ep \in [0,1]$, we assume that $u^\ep: \mathcal B\times [0,T] \to \R$ is such that 
\beq\label{scheme}
\bc
u^\ep (\cdot,t)=u_0 \ \text{if} \  t \in [0,\ep], \\[2mm]
u^\ep(\cdot, t)= S(\ep)u^\ep(\cdot, t -\ep) \ \text{if} \ t \in (\ep, T]. 
\ec
\eeq
The convergence result is stated next.

\begin{theorem}\label{thm1}
Assume \eqref{id}, \eqref{translation}, \eqref{monotonicity},  \eqref{consistency}, and \eqref{lh1}, and, for $u_0\in \textup{BUC}(\R^d)$, let $u^\ep$ be defined by \eqref{scheme}. 
Then, as $\ep \to 0$, $u^\ep  \to u \in \textup{BUC}(\R^d\times [0,T])$ locally uniformly in $\R^d\times [0,T]$, which is  
the unique solution of \eqref{ivp1}.
\end{theorem}

The proof of the theorem follows closely the arguments of the analogous theorem of \cite{BS}, thus we only sketch it next.
\begin{proof}[Sketch of proof of Theorem~\ref{thm1}]
Since $u_0$ is bounded, it  follows from \eqref{id}, \eqref{translation}, and \eqref{monotonicity} that the $u^\ep$'s are also bounded independently of $\ep$ in $\R^d\times [0,T]$. 
\smallskip

Hence, the local uniform upper and lower limits $u^\star \in \textup{BUSC}(\R^d\times [0,T])$ and $u_\star \in \textup{BLSC}(\R^d\times [0,T])$  of the $u^\ep$'s given respectively by 
\beq\label{limsup}
u^\star (x,t)=\underset{(y,s) \to (x,t), \ep \to 0} \limsup u^\ep(y,s)  \ \ \text{and} \ \ u_\star (x,t)=\underset{(y,s) \to (x,t), \ep \to 0} \liminf u^\ep(y,s)
\eeq
are well-defined. 
\vs

The goal is to show that $u^\star$ is a subsolution and $u_\star$ is a supersolution of \eqref{ivp1}. Then the assumed comparison principle, combined with the obvious inequality $u_\star\leq u^\star$, imply that 
$u^\star=u_\star$ and $u=u^\star=u_\star$ is the unique solution of \eqref{ivp1}, the latter being a consequence of the fact that $u$ is both subsolution and supersolution. 
The local uniform nature of the limits in  
\eqref{limsup} then yields the local uniform convergence of the $u^\ep$'s to $u$.
\vs
Since the arguments are similar, here we show only that $u^\star$ is a subsolution. For this, we assume that, for a given $\phi \in C^2(\R^d)$ and $g \in C^1(0,T]$, $(x_0,t_0) \in \R^d\times(0,T]$ is a strict global maximum of $u^\star-\phi -g$ in $\R^d\times(0,T]$. 
\vs

The definition of $u^\star$ and  some  calculus considerations  (see \cite[Proposition~4.3]{CIL}) yield $\ep_n\to 0$ and 
$(x_n,t_n) \in \R^d\times (\ep_n,T]$ such that $(x_n,t_n) \to (x_0,t_0)$, $u^{\ep_n}(x_n,t_n)\to u^\star(x_0,t_0)$ and $u^{\ep_n}-\phi-g$  achieves a global maximum at $(x_n,t_n)$.
\vs

\smallskip

Since, in view of \eqref{scheme}, 
\[u^{\ep_n}(x_n,t_n)=S(\ep_n)u^{\ep_n}(\cdot,t_n-\ep_n)(x_n)\]
and, for all $x\in \R^d$, 
\[u^{\ep_n}(x,t_n-\ep_n)-u^{\ep_n}(x_n,t_n)\leq \phi(x)-\phi(x_n) +g(t_n-\ep_n)-g(t_n),\]
\vs
it follows from the equivariance of translations by constants and the monotonicity of $S$ that 
\[S(\ep_n)[\phi-\phi(x_n)](x_n) +g(t_n-\ep_n)-g(t_n)\ge S(\ep_n) [u^{\ep_n}(\cdot,t_n-\ep_n)- u^{\ep_n}(x_n, t_n)](x_n)=0,\]
and, hence, 
\[
g'(t_n)\ep_n \le  S(\ep_n)[\phi-\phi(x_n)](x_n)+ \text{o}(\ep_n).
\]
\vs

Dividing by $\ep_n$ and letting $n\to \infty$ yields the subsolution property. The other inequality follows similarly.
\end{proof}

\section{Hyperbolic scaling examples}\label{hyperbolicsec}
We study here two models of deterministic surface growth with non-smooth driving function which are related to the deterministic version of directed last-passage percolation with driving function $\varphi((u_a)_{a\in A}) = \max_{a\in A} u_a$, which is the multidimensional analogue of a one-dimensional deterministic growth model considered by Krug and Spohn~\cite{ks88}.

\vs

We begin with the directed last-passage percolation model, which, at the discrete level, is given, for $(x,t)\in \Z^d\times\Z_+$, by 
\[
u(x, t+1)=\varphi ((u(x+a,t))_{a\in A}) = \max_{a\in A}  \ u(x+a,t).
\]

It is immediate that, if $u: \Z^d \to \R$ is linear, that is, $u(x)=p\cdot x$ for some $p\in \R^d\setminus \{0\}$, then 
\[ \underset {a\in A} \max \ p\cdot (x+a)= p \cdot x + \underset {a\in A} \max \ p\cdot a=p\cdot x+ \max_{i=1,\ldots,d} |p_i| \neq 0. \]
 
 Thus it is appropriate to use scale hyperbolically.  Note that  $\varphi (0)=0$. Moreover, the scheme is obviously equivariant under translations by constants and monotone, that is, for all $u, v: \Z^d\to \R$ such that $u\leq v$ and $k\in \R$, 
 \[\varphi ((u(x+a) +k)_{a\in A})=\varphi ((u(x+a))_{a\in A}) +k \ \ \text{and} \ \  \varphi ((u(x+a))_{a\in A})\leq \varphi ((v(x+a))_{a\in A}).\]
 
\vs
Following section~\ref{setup}, for $\ep>0$, $u\in \mathcal B$, and $x\in \R^d$, we define
\[
S(\ep)u(x)=\varphi ((u(x+\ep a))_{a\in A}),
\]
and note that   \eqref{id}, \eqref{translation} and \eqref{monotonicity} are satisfied.
\vs

Next we  check  the consistency. We fix $\phi$ smooth and look at the $\ep\to 0$ and $y\to x$ limit of  
\[\dfrac{S(\ep)\phi(y)-\phi (y)}{\ep}= \dfrac{\varphi ( (\phi(y+\ep a))_{a\in A})-\phi (y)}{\ep}.\]
It is immediate that
 \[
 \begin{split}
 \dfrac{\varphi ((\phi(x+\ep a))_{a\in A})-\phi (x)}{\ep}&= \frac{1}{\ep} \max_{a\in A} (\phi(y+\ep a)-\phi (y))\\[1.5mm]
 &= \frac{1}{\ep}\max_{i=1,\ldots, d} (\pm \phi_{x_i}(y) \ep + \text{o}(\ep^2))\\[1.5mm]
 &=   \underset {i=1,\ldots, d}\max |\phi_{x_i} (y)| + \text{o}(\ep) \underset{y\to x, \ep \to 0}\to \underset {i=1,\ldots, d} \max |\phi_{x_i} (x)|.
 \end{split}
 \]

Let $H=H(p):\R^d\to \R$ be given by 
\beq\label{takis1202.2}
H(p)= \underset {i=1,\ldots, d} \max |p_i|.
\eeq
 
Since $H\in C(\R^d)$, it follows from \cite{CIL}, that, for each $T>0$, the initial value problem
\beq\label{takis1203}
u_t=H(Du) \ \text{in} \ \R^d \times [0,T] \ \ \ u(\cdot,0)=u_0
\eeq  
admits a comparison principle in $\text{BUC}(\R^d\times [0,T])$, and, thus,  \eqref{lh1} is satisfied. Moreover, for each $u_0\in \text{BUC}(\R^d)$, \eqref{takis1203} has a unique solution $u\in \text{BUC}(\R^d\times [0,T])$.
\vs

We have proved the following theorem. 
\begin{theorem}\label{hyper1}
If $u^\ep:\R^d\times [0,T] \to \R$ is defined by \eqref{takis1104.1} starting with $u_0\in \textup{BUC}(\R^d)$,  then, as $\ep\to 0$ and locally uniformly in $\R^d\times [0,T]$, $u^\ep\to u$, the 
unique solution $u\in \textup{BUC}(\R^d\times [0,T])$ of \eqref{takis1203} with $H$ as in \eqref{takis1202.2} and initial data $u_0$.
\end{theorem}

As mentioned above, this is also a scheme proposed in \cite{ks88} to obtain at the limit a deterministic KPZ-type nonlinearity with sublinear growth, that is, the PDE
\[u_t=\Delta u +|Du|.\]
It is clear, in view of Theorem~\ref{hyper1}, that  this conjecture in \cite{ks88} is not possible even when $d=1$ in which case, of course, $H(p)=|p|.$ 

\vs

We consider next  another scheme which, at the discrete level, is given, for $(x,t)\in \Z^d\times\Z_+$, by 
\beq\label{1300}
u(x, t+1)=\varphi ((u(x+a,t))_{a\in A})= u(x,t) + \df{1}{2d}\underset {a\in A} \sum (u(x+a,t)-u(x,t))_+.
\eeq
This is a variant of example (1.5) from \cite{C}, where $(u(x+a)-u(x))_+$ is replaced by $q(u(x+a)-u(x))$ for an increasing  $C^2$ function $q$. Parabolic scaling applies in that case, and the limit is the deterministic KPZ equation. As we will see below, that is no longer the case for this variant, a fact that stresses the consequences of the lack of differentiability of the driving function.

\vs 
The fact that  $\varphi (0)=0$ as well as the equivariance under translations by constants and the monotonicity are immediate. 
\vs

As above, for $\ep>0$ and $u\in \mathcal B$, we define
\[S(\ep)u(x)= \varphi ((u(x+\ep a))_{a\in A}).\]

It is immediate that   \eqref{id}, \eqref{translation} and \eqref{monotonicity} are  satisfied.
\vs

To check  \eqref{consistency}, we fix  $\phi:\R^d\to \R$ smooth and look at the $\ep \to 0$ and $y\to x$ limit of 
\[
\df{S(\ep)\phi(y)-\phi(y)}{\ep}=\df{\varphi ((\phi(y+\ep a))_{a\in A})-\phi (y)}{\ep}.
\]

A straightforward computation and \eqref{1300} yield the following string of equalities and limits.
\[
\bs
\df{\varphi ((\phi (y+\ep a))_{a\in A})-\phi (y)}{\ep} &=
\df{\phi(y) + \df{1}{2d}\underset {a\in A} \sum (\phi(y+\ep a)-\phi(y))_+ -\phi(y)}{\ep}\\[1.5mm]
&=\df{1}{2d}\underset {a\in A} \sum (D\phi(y)\cdot a +\text{o}(\ep))_+ \ \underset{y\to x, \ep\to 0}\to \ \df{1}{2d} \sum_{i=1}^d (\phi_{x_i} (x))_+.
\end{split}
\]

Let $H=H(p):\R^d\to \R$ be given by 
\beq\label{takis1302}
H(p)=\df{1}{2d} \sum_{i=1}^d (p_i)_+.
\eeq

Since $H\in C(\R^d)$, the initial value problem 
\[
u_t=H(Du) \ \text{in} \ \rt, \ \ \ u(\cdot,0)=u_0
\]
admits a comparison principle in $\text{BUC}(\R^d\times [0,T])$, and, thus,  \eqref{lh1} is satisfied. Moreover, for each $u_0\in \text{BUC}(\R^d)$, \eqref{takis1203} has a unique solution $u\in \text{BUC}(\R^d\times [0,T])$.
\vs

We thus proved the following convergence result. 
\begin{theorem}\label{hyper3}
For  $u_0\in \textup{BUC}(\R^d)$ and let  $u\in \textup{BUC} (\R^d\times [0,T])$ be the unique solution of \eqref{takis1203} with $H$ as in \eqref{takis1302} and initial data $u_0$. If $u^\ep:\R^d\times [0,T] \to \R$ is defined by \eqref{takis1104.1}, then, as $\ep\to 0$ and locally uniformly in $\R^d\times [0,T]$, $u^\ep\to u$.
\end{theorem}

\section{Parabolic scaling examples}\label{parabolicsec}

We divide this section into two subsections depending on the regularity and properties of the generating function $\varphi$. In the first we generalize and give a much simpler proof of  the result of \cite{C}. The second is about new results concerning zero temperature dynamics of gradient Gibbs measures.  This subsection is also divided into three parts depending on the behavior of the underlying potential. 

\subsection{Generalization of deterministic KPZ-type models}
We provide  an extension of the main result of \cite{C}  by obtaining a generalized deterministic KPZ scaling limit under the assumptions that the height function  is equivariant under constant shifts, monotone, and twice continuously differentiable. 
The class of PDEs obtained in the limit contain as a very particular case the classical deterministic KPZ equation. 
\vs

We assume that the evolution of a deterministically growing $d$-dimensional surface 
$u:  \Z^d \times \Z_+ \to \R$  is   as in \eqref{takis1100} with $\varphi$ given, for each $x \in  \Z^d$ and $v:\Z^d\to \R$, by  
\[
\varphi((v(x+a))_{a\in A})= v(x) + \Phi (v(x \pm e_1) -v(x), \ldots, v(x \pm e_d) -v(x)),
\]
where 
\beq\label{ex2}
\Phi \in C^2(\R^{2d}) \ \ \text{and} \ \ \Phi(0,\ldots,0)=0.
\eeq
The assumption that $\Phi(0,\ldots,0)=0$ is made only to simplify the presentation, since, as discussed earlier, we can always work with 
$\tilde \Phi =\Phi -\Phi(0,\ldots,0)$. We leave the details to the reader. 
\vs

We also assume that, for each $i=1,\ldots,d$,
\beq\label{ex1.5ab}
\Phi_{v_i}(0,\ldots,0) = \Phi_{v_{-i}}(0,\ldots,0),
\eeq
where $\Phi_{v_i}$ is shorthand for $\partial \Phi/\partial v_i$.
This is a relaxation of the ``invariance under lattice symmetries'' assumption from \cite{C}. 
\vs

We note that \eqref{ex1.5ab} is necessary to allow us to consider a parabolic scaling. As discussed earlier, this is related to the fact that for such a scaling we need to have, for all $p\in \R^d$, 
\[\lim_{\ep \to 0} \df{\varphi((p\cdot (x+\ep a))_{a\in A})-p\cdot x}{\ep}=0,\]
which follows from \eqref{ex1.5ab}. 
\vs

After rescaling and using the setup discussed in Subsection~\ref{setup} we define the scheme 
\beq\label{ex1.1}
S(\ep)v(x)= v(x) +   \Phi (v(x \pm \sqrt {\ep} e_1) -v(x), \ldots, v(x \pm \sqrt {\ep} e_d) -v(x)).
\eeq

It is immediate that 
$$ S(0)v=v \ \text{and} \ S(\ep)(v+k)= S(\ep)v +k,$$ 
and, hence, \eqref{id} and \eqref{translation} are  satisfied. 
\smallskip

For the monotonicity, it is enough to show that the map 
\[
(v_0, v_{\pm1},\ldots, v_{\pm d}) \mapsto  \tilde S(v_0, v_{\pm1},\ldots, v_{\pm d}):=v_0 + \Phi (v_{\pm 1}-v_0, \ldots, v_{\pm d}-v_0)
\]
is monotone with respect to all its argument. And for this, we need that for all $(v_0, v_{\pm1},\ldots, v_{\pm d})$ and $ i=1,\ldots, d ,$ 
\[\dfrac{\partial \tilde S}{\partial v_0}(v_0, v_{\pm1},\ldots, v_{\pm d})\geq 0 \ \ \text{and} \ \ \dfrac{\partial \tilde S}{\partial v_{\pm i}}(v_0, v_{\pm1},\ldots, v_{\pm d}) \geq 0. \]

It is immediate that 
\[\dfrac{\partial \tilde S}{\partial v_0}(v_0, v_{\pm1},\ldots, v_{\pm d})=1 - \DIV \Phi (v_{\pm 1}-v_0, \ldots, v_{\pm d}-v_0), \]
and, for each $i=1,\ldots,d$,   
\[\dfrac{\partial \tilde S}{\partial v_{\pm i}}(v_0, v_{\pm1},\ldots, v_{\pm d})=  \Phi_{v_{\pm i}}(v_{\pm1}-v_0,\ldots, v_{\pm d}-v_0).\]
 
Thus, the monotonicity of $S$ at large is equivalent to the assumptions that, for  all $v_{\pm1},\ldots, v_{\pm d}\in \R$,
\beq\label{ex1.3a}
1-\DIV \Phi(v_{\pm1},\ldots, v_{\pm d}) \ge 0 \ \ \text{and} \ \  \Phi_{v_{\pm i}} (v_{\pm1},\ldots, v_{\pm d})\geq 0. 
\eeq

As an aside, we remark that instead of  \eqref{ex1.3a},  we may assume 
\beq\label{ex1.3ab}
1-\DIV \Phi(0,\ldots, 0) > 0 \ \ \text{and} \ \  \Phi_{v_{\pm i}} (0,\ldots, 0) >0. 
\eeq
Indeed, recall that \eqref{contraction} preserves the Lipschitz continuity of the scheme. Thus, if we assume that $u_0$ is bounded and Lipschitz continuous, then the scheme generates a bounded and Lipschitz continuous $u^\ep$. It is then immediate that, for all $i=1,\ldots,d$,
\beq\label{takisps1}
u^\ep(x \pm \sqrt\ep e_i,t)- u^\ep(x) =\text{O}(\sqrt \ep),
\eeq
with $\text{O}$ depending only on the Lipschitz constant of $u_0$.
\vs  
If $\phi:\R^d\to \R$ is a smooth function, we also have, for all $i=1,\ldots,d$,
\beq\label{takisps2}
\phi(x\pm \sqrt\ep e_i,t)- \phi(x) =\text{O}(\sqrt \ep),
\eeq
with $\text{O}$ depending only on the Lipschitz constant of $\phi$.
\vs
Looking back at the proof of Theorem~\ref{thm1} we see that the monotonicity of the scheme is only used to replace terms like $u^\ep(x\pm \sqrt\ep e_i,t)- u^\ep(x,t) $ by terms like $\phi(x\pm \sqrt\ep e_i,t)- \phi(x)$, which, in view of \eqref{takisps1} and \eqref{takisps2}, are of (uniform) order $\sqrt \ep$. It suffices then to have the monotonicity of the $\Phi$ only in a $\sqrt \ep$-neighborhood of $0$. Hence, it suffices to assume \eqref{ex1.3ab}. 
\vs
Once we have convergence for bounded and Lipschitz continuous $u_0$, the result for $u_0\in \text{BUC}(\R^d)$ follows again from the contraction property of the scheme and the limit problem and an elementary density argument.

\vs

To make what follows easier to read, let
\beq\label{Phidef}
\Phi^{\ep}(y) := \Phi(\phi(y \pm \sqrt{\ep}e_1) - \phi(y), \ldots, \phi(y \pm \sqrt{\ep} e_d) - \phi(y))=  S(\ep)\phi(y) - \phi(y).
\eeq

Taylor's  expansion, for $\ep$ small  and uniformly in $y$ near  $x$, gives 
\begin{align*}
\Phi^{\ep}(y) &= \sum_{i=1}^d[ \Phi_{v_i}(0,\ldots,0) (\phi(y + \sqrt{\ep}e_i) - \phi(y)) + \Phi_{v_{-i}}(0,\ldots,0) (\phi(y - \sqrt{\ep}e_i) - \phi(y))]\\
&\qquad + \frac{1}{2} \sum_{i=1}^d\sum_{j=1}^d \Phi_{v_iv_j}(0,\ldots,0) (\phi(y + \sqrt{\ep}e_i) - \phi(y))(\phi(y + \sqrt{\ep}e_j) - \phi(y))\\
&\qquad + \frac{1}{2} \sum_{i=1}^d\sum_{j=1}^d \Phi_{v_{-i}v_{-j}}(0,\ldots,0) (\phi(y - \sqrt{\ep}e_i) - \phi(y))(\phi(y -  \sqrt{\ep}e_j) - \phi(y))\\
&\qquad +  \sum_{i=1}^d\sum_{j=1}^d \Phi_{v_iv_{-j}}(0,\ldots,0) (\phi(y + \sqrt{\ep}e_i) - \phi(y))(\phi(y - \sqrt{\ep}e_j) - \phi(y)) + \text{O}(\ep^{3/2}), 
\end{align*}
and 
\begin{align*}
\phi(y \pm \sqrt{\ep}e_i) - \phi(y) &= \pm \sqrt{\ep}\phi_{x_i}(y) + \frac{\ep}{2}\phi_{x_ix_i}(y) + \text{O}(\ep^{3/2}). 
\end{align*}
Substituting the last expression  in the previous display, and using \eqref{ex1.5ab}, we get
\begin{align*}
&\lim_{y\to x, \ep \to 0} \df{\Phi^{\ep}(y)}{\ep} =  \sum_{i=1}^d \Phi_{v_i}(0,\ldots,0)\phi_{x_ix_i}(x) + \frac{1}{2} \sum_{i=1}^d\sum_{j=1}^d \Phi_{v_iv_j}(0,\ldots,0) \phi_{x_i}(x)\phi_{x_j}(x)\\
&\qquad + \frac{1}{2} \sum_{i=1}^d\sum_{j=1}^d \Phi_{v_{-i}v_{-j}}(0,\ldots,0) \phi_{x_i}(x)\phi_{x_j}(x) -  \sum_{i=1}^d\sum_{j=1}^d \Phi_{v_iv_{-j}}(0,\ldots,0) \phi_{x_i}(x)\phi_{x_j}(x).
\end{align*}
Note that, if \eqref{ex1.3a} is assumed,  then the map 
\beq\label{ex1.6a1}
M \mapsto \sum_{i=1}^{d}  \Phi_{v_i}(0,\ldots,0) M_{ii}  \ \ \text{is  degenerate  elliptic  in $\mathcal S^d$,}
\eeq
while, when \eqref{ex1.3ab} holds, 
\beq\label{ex1.6a}
M \mapsto \sum_{i=1}^{d}  \Phi_{v_i}(0,\ldots,0) M_{ii} \ \ \text{is  uniformly elliptic in $\mathcal S^d$.}
\eeq

Let the Hamiltonian $H=H(p)=H(p_{1},\ldots, p_{d}):\R^{d}\to \R$ and the matrix 
 $A=(A_{ij})_{i,j=1,\ldots, d}\in \mathcal{S}^d$ be defined by 
\beq\label{ex1.81}
\begin{split}
A_{ij}&=\delta_{ij}\Phi_{v_i}(0,\ldots,0) \ \text{and} \\
H(p)&=\dfrac{1}{2}\sum_{i=1}^d\sum_{j=1}^d \Big(\Phi_{v_iv_j}(0,\ldots,0) + 
\Phi_{v_{-i}v_{-j}}(0,\ldots,0) - 2 \Phi_{v_iv_{-j}}(0,\ldots,0)\Big) p_ip_j.
\end{split}
\eeq
Clearly $H\in C(\R^d)$ and, in view of \eqref{ex1.6a1} and \eqref{ex1.6a},  the matrix $A$ is either degenerate  elliptic or uniformly elliptic. It then follows from \cite{CIL} that the initial value problem  
\beq\label{ex1.9a}
u_t=\text{trace}(AD^2u) + H(Du) \  \text{in} \ \R^d\times (0,T] \ \ \ u(\cdot,0)=u_0
\eeq
has, for every $u_0\in \text{BUC}(\R^d)$, a unique solution $u\in \text{BUC}(\R^d\times [0,T])$, which is classical for $t>0$ if \eqref{ex1.3ab} holds, and, moreover, satisfies \eqref{lh1}. 
\vs

Note that the well-posedeness of \eqref{ex1.9a} is a standard fact in the theory of viscosity solutions, hence, we omit the details and instead we refer to Theorem~8.2 in \cite{CIL}. 
\vs

Collecting all the hypotheses above and using Theorem~\ref{thm1} we have now the following theorem which extends the 
corresponding result in \cite{C}. 

\begin{theorem}\label{thm2}
Assume \eqref{ex2},  \eqref{ex1.5ab}, and either \eqref{ex1.3a} or  \eqref{ex1.3ab}. Then the scheme defined using \eqref{ex1.1} converges locally uniformly, as $\ep \to 0$, to the unique solution of \eqref{ivp1} with $F$ as in \eqref{ex1.81}.
\end{theorem}

\subsection{Zero temperature dynamics of gradient Gibbs measures}\label{gibbs}
In this subsection, we study  schemes  generated by  a different class of growth models in which the location of the growing surface is determined by minimizing the total potential energy between the point and its neighbors. This is motivated by Glauber dynamics for gradient Gibbs measures at zero temperature. 
\vs

Formally, a gradient Gibbs measure is probability measure on $\R^{\Z^d}$ with probability density proportional to 
\begin{align*}
\exp\biggl(-\beta \sum_{x,y\in \Z^d, |x-y|=1} V(h(x)-h(y))\biggr),
\end{align*}
where  $\beta$ is the inverse temperature parameter, and  $V$  is a potential function, often assumed to be convex and symmetric. We will assume that $V$ is convex and symmetric throughout this subsection. For background on gradient Gibbs measures, see \cite{bk07, sheffield05}.
\vs

The {Glauber dynamics} for a gradient Gibbs measure as above proceeds by updating the height $h(x)$ at a site $x$ by regenerating $h(x)$ from the conditional distribution given the heights at neighboring points. When the temperature is zero, or, in other words, $\beta=\infty$, the Glauber dynamics simply chooses $h(x)$ that minimizes 
\[
\sum_{i=1}^d V(h(x)-h(x\pm e_i)). 
\]
If $V$ has a strict  minimum  and is  differentiable there, then the minimizing problem has a unique solution. Hence, the zero temperature dynamics becomes fully deterministic. We will investigate this below. 
\vs
When $V$ either has a strict minimum but is not differentiable there, or the minimum is not strict, the set of minimizers may be an interval of positive length. In this case, the zero temperature dynamics will choose a point uniformly from this interval. To avoid this randomness, we will simply choose the midpoint of the interval as the updated height. We discuss two such examples later.
\vs

In the language of Section \ref{setup}, the driving function $\varphi:\R^A\to \R$  is 
\[
\varphi((u_a)_{a\in A})= \text{argmin}_{y\in \R}  \sum_{a\in A} V(y-u_a),
\]
where 
\beq\label{v.a}
V:\R\to \R \ \text{is a symmetric convex potential with minimum  $0$ at $y=0$.}
\eeq

As mentioned above, depending on the properties of $V$, for each $(u_a)_{a\in A}$,  the minimum value of the map $y\mapsto   \sum_{a\in A} V(y-u_a)$ can be achieved at either a single point or a closed interval. In the former case, $\text{argmin}$ has its usual meaning. In the latter, $\text{argmin}$ is taken to be the middle point of the interval of minima.
\vs  

Next, we show that parabolic scaling is the correct scaling to study the asymptotic behavior. For this, we need to show that, for any $p\in \R^d$, as $\ep\to 0$,
\beq\label{takis1310}
\df{\text{argmin}_{y\in \R} \sum_{i=1}^d V(y-p\cdot (x \pm \ep e_i)) -p\cdot x}{\ep} \to 0.
\eeq

To show \eqref{takis1310}, we begin with the elementary observation that, for any  $v:\R^d \to \R$,
\[\text{argmin}_{y\in \R} \left[ \sum_{i=1}^d V(y-v(x \pm \ep e_i))\right] -v(x)= \text{argmin}_{y\in \R} \sum_{i=1}^d V(y-(v(x \pm \ep e_i)-v(x))).\]

Next we observe that, since $V$ is convex and even, so is the map 
\[y\mapsto \sum_{i=1}^d V(y-(p\cdot (x \pm \ep e_i)-p\cdot x ))=\sum_{i=1}^d V(y\pm \ep p_i).\]
and thus 
\[\text{argmin}_{y\in \R}  \sum_{i=1}^d V(y-p\cdot (x \pm \ep e_i)) -p\cdot x=0.\]
Finally, to prove the convergence we introduce the maps $(S(\ep))_{\ep\in [0,1]}$, which, for $u\in \mathcal B(\R^d)$ and $\ep>0$, is given by 
\beq\label{ex2.2}
S(\ep)\phi(x)=\phi(x) + \left[\Phi (\phi(x\pm \sqrt{\ep} e_1),\ldots, \phi(x\pm \sqrt{\ep} e_d)) - \phi(x)\right],
\eeq
where, given $v_{\pm i}$ for $i=1,\ldots,d$, 
\[
\Phi (v_{\pm1}, \ldots, v_{\pm d})= \text{argmin}_{y \in \R} \left[ \sum_{i=1}^d (V(y -v_{-i}) + V(y-v_{i}))\right].
\]

In view of \eqref{v.a}, it is immediate that $S(\ep)$ satisfies \eqref{id}, \eqref{translation} and \eqref{monotonicity}. 
\vs

The consistency is more complicated and depends on the regularity of $V$ at $0$ and on whether 
$0$ is a strict minimum or not. There are three different cases and we study each one separately.
\vs

\subsubsection{\bf ``Smooth'' potentials with strict minimum}  In addition to \eqref{v.a},  we assume that 
\beq\label{vsmooth.aa}
\bc
V \ \text{is twice differentiable in a neighborhood of  $0$, which is a strict minimum, and}\\[1.2mm]
\text{there exists $\sigma >0$ and $\alpha \neq 0$ such that $\underset{|y|\to 0}\lim \dfrac{V^{(2)}(y)}{|y|^\sigma}=\alpha,$} 
\ec
\eeq
where, for $l\in \Z_+$, $V^{(l)}$ denotes the $l$-th derivative if it exists.
\vs
For future reference, we record here the elementary fact that, in view of \eqref{vsmooth.aa}, for each $\theta>0$ and $|y|$ sufficiently small,
\beq\label{hp5}
(\alpha -\theta) |y|^\sigma \leq V^{(2)}(y) \leq (\alpha + \theta)  |y|^\sigma. 
\eeq
An example of a potential satisfying  \eqref{vsmooth.aa} is a  $V:\R\to [0,\infty)$ such that  \beq\label{vsmooth1aa}
\bc
\text{there exists an even number $k\ge 2$ such that $V\in C^{k+1}$,  $V^{(k)}(0)>0$}, \\[1.5mm] 
\text{and $V^{(1)}(0)=\cdots=V^{(k-1)}(0)=0.$ }
\ec
\eeq
Since $V$ is even, all the odd derivatives at $0$, if they exist, vanish. Hence, \eqref{vsmooth1aa} is about the existence of the $k$. 
\vs
Another example is the potential 
\beq\label{athens11}
V(x)=|x|^{2+\delta} \ \  \text{for some $\delta \in (0,1)$.} 
\eeq
In view of the definition of the scheme generating $\Phi$, for $\phi$ smooth, we have 
\[
\sum_{i=1}^d V^{(1)}( \Phi(\phi (x \pm \sqrt{\ep} e_1)-\phi(x), \ldots, \phi (x \pm \sqrt{\ep}e_d)-\phi(x))=0
\]
In what follows, to simpliy the notation, we write, as in \eqref{Phidef},
\begin{align*}
\Phi^{\ep}(y) := \Phi (\phi (y \pm \sqrt{\ep} e_1)-\phi(y), \ldots, \phi (y \pm \sqrt{\ep}e_d)-\phi(y)),
\end{align*}
We study next the consistency of the scheme which follows from the asymptotic behavior of $\ep^{-1}\Phi^{\ep}(y)$ as $\ep \to 0$ and $y$ converges to some $x$. The main result needs the following lemma, which is proven at the end of the ongoing subsection.
\vs  
\begin{lemma}\label{takis11111}
 Assume \eqref{v.a} and \eqref{vsmooth.aa}. Then, for any $x\in \R^d$, 
 \[
\frac{1}{2}\min_i \phi_{x_ix_i}(x)\le \liminf_{y\to x, \ep \to 0} \frac{\Phi^{\ep}(y)}{\ep}  \le \limsup_{y\to x, \ep \to 0} \frac{\Phi^{\ep}(y)}{\ep}\le \frac{1}{2}\max_i \phi_{x_ix_i}(x).
 \]
Moreover, if  $D\phi(x)\neq 0$, then
 \beq\label{ex2.2b.aa}
 \lim_{y\to x, \ep \to 0} \frac{\Phi^{\ep}(y)}{\ep} =  \dfrac{\sum_{i=1}^d |\phi_{x_i}(x)|^{\sigma} \phi_{x_ix_i}(x)}{ 2\sum_{i=1}^d |\phi_{x_i}(x)|^{\sigma}}.
\eeq
\end{lemma}
When the potential is as in \eqref{vsmooth1aa}, then \eqref{ex2.2b.aa} reads as
\[\lim_{y\to x, \ep \to 0} \frac{\Phi^{\ep}(y)}{\ep} =  \dfrac{\sum_{i=1}^d |\phi_{x_i}(x)|^{k-2} \phi_{x_ix_i}(x)}{ 2\sum_{i=1}^d |\phi_{x_i}(x)|^{k-2}},\]
while, if $V$ is as in \eqref{athens11},
\[\lim_{y\to x, \ep \to 0} \frac{\Phi^{\ep}(y)}{\ep} =  \dfrac{\sum_{i=1}^d |\phi_{x_i}(x)|^{\delta} \phi_{x_ix_i}(x)}{ 2\sum_{i=1}^d |\phi_{x_i}(x)|^{\delta}}.\]
\vs
Next we introduce the equation satisfied by the limit of the scheme.
\vs
Let $F:\mathcal{S}^d\times (\R^d\setminus \{0\}) \to \R$ be given by 
\beq\label{ex2.3.aa}
F(X,p)= \dfrac{\sum_{i=1}^d |p_i|^{\sigma} X_{ii}}{ 2\sum_{i=1}^d |p_i|^{\sigma}}.
\eeq
It is immediate that $F\in C(\mathcal{S}^d\times (\R^d\setminus \{0\}))$ is degenerate elliptic. Moreover, since, in view of \eqref{ex2.3.aa}, for each $p\ne 0$ and each $X$, $F(X,p)$ is a weighted average of $X_{11},\ldots,X_{dd}$, it is immediate that if  $(Y_n,q_n)\in\mathcal{S}^d \times (\R^d\setminus\{0\}) \to (X,0)$, then 
\[\frac{1}{2} \min_i X_{ii} \le  \liminf _{n\to \infty}F(Y_n,q_n) \le \limsup_{n\to \infty} F(Y_n,q_n) \le \frac{1}{2} \max_i X_{ii}.\]
Finally, with appropriate choices of $q_n$, it is easy to see that equality can be attained in both cases.
Thus, for $F^\star$ and $F_\star$ defined as in \eqref{fstar1} and \eqref{fstar2}, we have
\beq\label{newFstar}
F^\star(X,0)= \frac{1}{2} \max_i X_{ii} \ \ \text{ and } \ \ F_\star(X,0)=\frac{1}{2} \min_i X_{ii}.
\eeq
It follows from Theorem~5 in \cite{MS} that the initial value problem  \eqref{general_ivp1}, with 
\beq\label{chicago1102} \overline F=\underline F=F \ \ \text{ in} \ \ \mathcal{S}^d\times (\R^d\setminus \{0\}), \eeq
and 
\beq\label{chicago1103} \overline F(X,0)= F^\star(X,0) \ \ \text{ and} \ \ \underline F(X,0)= F_\star(X,0) \ \ \text{ for}  \ \ X\in \mathcal{S}^d,\eeq
admits a comparison principle, that is, \eqref{lh1} is satisfied. 
\vs
We may now apply Theorem~\ref{thm1} to state our second main result. 
\begin{theorem}\label{thm3}
Assume  \eqref{v.a} and \eqref{vsmooth.aa}.  Then, for every $u_0 \in \textup{BUC}(\R^d)$, the scheme defined by \eqref{ex2.2} converges, as $\ep \to 0$ and locally uniformly in $\R^d \times [0,T]$, to the unique solution of \eqref{general_ivp1} with $F$ given by  \eqref{chicago1102} and \eqref{chicago1103}.
\end{theorem}
We continue with the proof of the previous lemma.
\vs
\begin{proof}[The proof of Lemma~\ref{takis11111}] 
Throughout the arguments below, $\ep$ is small enough and $y$ is close enough to $x$.  Moreover, all the limits are taken 
with $\ep\to 0$ and $y\to x$. Both  facts will not be repeated from step to step.
\vs

In addition, to ease the notation, for $i=1,\ldots,d$, we  set 
\beq\label{hp10}
a_i(y) = \phi_{x_i}(y)  \ \ \text{and} \ \ b_i(y) = \dfrac{1}{2}\phi_{x_ix_i}(y).
\eeq

It follows that 
\begin{align}\label{phiab}
\phi (y \pm \sqrt{\ep} e_i) - \phi(y)  = \pm \sqrt{\ep}a_i(y) + \ep b_i(y) + \text{o}(\ep)=  \pm \sqrt{\ep}a_i(y) + \ep b^\ep_i(y) 
\end{align}
(where $\ep b^\ep_i(y) = \ep b_i(y)+\text{o}(\ep)$) and  
\begin{align}\label{laplacian}
\phi(y+\sqrt{\ep} e_i) + \phi(y-\sqrt{\ep}e_i) - 2\phi(y) = 2\ep b_i(y) + \text{O}(\ep^{3/2}).
\end{align}
Moreover,  
\[
a_i(y)\to a_i(x) \ \ \text{and} \ \  b^\ep_i(y) \to b_i(x).
\]

\vs
For $z\in \R$, we set
\[
f^{\ep,y}(z) = \sum_{i=1}^dV(\phi (y + \sqrt{\ep} e_i) - \phi(y) - z) +  \sum_{i=1}^dV(\phi (y - \sqrt{\ep} e_i) - \phi(y) - z).
\]
Then $\Phi^{\ep}(y)$ is a minimizer of the convex function $f^{\ep,y}$, and, hence,
\beq\label{chicago11112}
(f^{\ep,y})'(\Phi^{\ep}(y))=0.
\eeq
\vs
Since $V'$ is an odd function, we find
\beq
\begin{split}\label{fprime1}
(f^{\ep,y})'(z) &=  -\sum_{i=1}^dV'(\phi (y + \sqrt{\ep} e_i) - \phi(y) - z) -  \sum_{i=1}^dV'(\phi (y - \sqrt{\ep} e_i) - \phi(y) - z)\\[1.5mm]
&=  \sum_{i=1}^d[V'(z - \phi (y + \sqrt{\ep} e_i) + \phi(y)) - V'(\phi (y - \sqrt{\ep} e_i) - \phi(y) - z)]. 
\end{split}
\eeq
Suppose that 
\[
\begin{split}
z >& \max_i \frac{1}{2} (\phi(y+\sqrt{\ep} e_i) + \phi(y-\sqrt{\ep}e_i) - 2\phi(y))\\[1.5mm]
&(\text{resp.} \ \  z < \min_i \frac{1}{2} (\phi(y+\sqrt{\ep} e_i) + \phi(y-\sqrt{\ep}e_i) - 2\phi(y)),
\end{split}
\]
Then, for each $i$,
\[
\begin{split}
&z -  \phi (y + \sqrt{\ep} e_i) + \phi(y) > \phi (y - \sqrt{\ep} e_i) - \phi(y) - z\\[1.5mm]
&\qquad (\text{resp.} \ \ z - \phi (y + \sqrt{\ep} e_i) + \phi(y) < \phi (y - \sqrt{\ep} e_i) - \phi(y) - z ),
\end{split}
\]
and, hence, in view of  \eqref{fprime1} and the strictly increasing nature of $V'$ in a neighborhood of $0$ (which follows, e.g., from \eqref{hp5}), we get,  for $\ep$ small enough,  that
 $$(f^{\ep,y})'(z) >0 \ \ (\text{resp.} \ \ (f^{\ep,y})'(z) <0). $$ 
The convexity of $f^{\ep,y}$ and the observations above  imply
\beq\label{athens2.1}
\begin{split}
\min_i \frac{1}{2}(\phi(y+\sqrt{\ep} e_i) + \phi(y-\sqrt{\ep}e_i) - 2\phi(y)) \le \Phi^\ep(y)\\[1.5mm]
\le \max_i \frac{1}{2} (\phi(y+\sqrt{\ep} e_i) + \phi(y-\sqrt{\ep}e_i) - 2\phi(y)).
\end{split}
\eeq
The inequalities in \eqref{athens2.1}, together with \eqref{laplacian}, immediately imply the first claim of the lemma and the a priori estimate 
\begin{align}\label{apriori1}
\Phi^\ep(y)= \text{O}(\ep). 
\end{align}
%
To prove the second claim, we assume that $D\phi(x)\ne 0$, which yields  
that at least one of the $a_i(x)$'s is nonzero.
\vs

Using \eqref{chicago11112}, \eqref{fprime1}, and \eqref{phiab}, we find
\beq\label{hp1}
 \sum_{i=1}^d[V^{(1)}( \sqrt{\ep}a_i(y) +\Phi^\ep(y) -b_i^\ep(y)) - V^{(1)}( \sqrt{\ep}a_i(y) -(\Phi^\ep(y) -b_i^\ep(y)))]=0. 
\eeq
The $C^2$-regularity of $V$ yields,  for each $i=1,\ldots,d$, 
\beq\label{hp2}
\begin{split}
&V^{(1)}( \sqrt{\ep}a_i(y) +\Phi^\ep(y) -\ep b_i^\ep(y)) - V^{(1)}( \sqrt{\ep}a_i(y) -(\Phi^\ep(y) -\ep b_i^\ep(y))) \\&=2\int_0^1 V^{(2)}
(\sqrt{\ep}a_i(y) -(\Phi^\ep(y) -\ep b_i^\ep(y)) +2 \lambda (\Phi^\ep(y) -\ep b_i^\ep(y))) \; d\lambda\;  (\Phi^\ep(y) -\ep b_i^\ep(y)).
\end{split}
\eeq
Combining  \eqref{hp1} and \eqref{hp2} we get  
\[
\sum_{i=1}^d [\int_0^1 [V^{(2)}
(\sqrt{\ep}a_i(y) -(\Phi^\ep(y) -\ep b_i^\ep(y)) +2 \lambda (\Phi^\ep(y) -\ep b_i^\ep(y))) \; d\lambda\;  (\Phi^\ep(y) -\ep b_i^\ep(y))]=0,
\]
and, as long as 
\beq\label{hp8}
\sum_{i=1}^d  [\int_0^1 V^{(2)}
(\sqrt{\ep}a_i(y) -(\Phi^\ep(y) -\ep b_i^\ep(y)) +2 \lambda (\Phi^\ep(y) -\ep b_i^\ep(y))) \; d\lambda] \neq 0,
\eeq
we have
\[
\dfrac{\Phi^\ep(y)}{\ep}=\dfrac{\sum_{i=1}^d \int_0^1 [V^{(2)}
(\sqrt{\ep}a_i(y) -(\Phi^\ep(y) -\ep b_i^\ep(y)) +2 \lambda (\Phi^\ep(y) -\ep b_i^\ep(y))) \; d\lambda]\; b^\ep_i(y) }{\sum_{i=1}^d [\int_0^1 [V^{(2)}
(\sqrt{\ep}a_i(y) -(\Phi^\ep(y) -\ep b_i^\ep(y)) +2 \lambda (\Phi^\ep(y) -\ep b_i^\ep(y))) \; d\lambda] }.
\]

Next, we observe that \eqref{hp5} yields that, for each $i=1,\ldots,d$ 
and uniformly in $\lambda \in [0,1]$, 
\beq\label{hp32}
\begin{split}
(\alpha -\theta) |\sqrt{\ep}a_i(y) -(\Phi^\ep(y) -\ep b_i^\ep(y)) +2 \lambda (\Phi^\ep(y) -\ep b_i^\ep(y))|^\sigma \\[1.2mm]
\leq  V^{(2)}
(\sqrt{\ep}a_i(y) -(\Phi^\ep(y) -\ep b_i^\ep(y)) +2 \lambda (\Phi^\ep(y) -\ep b_i^\ep(y)))\\[1.2mm]
\leq  
(\alpha +\theta) |\sqrt{\ep}a_i(y) -(\Phi^\ep(y) -\ep b_i^\ep(y)) +2 \lambda (\Phi^\ep(y) -\ep b_i^\ep(y))|^\sigma.
\end{split}
\eeq 
Since  $|D\phi (x)| \neq 0$, we have
\[
\sum_{i=1}^d |a_i(x)|^\sigma >0,
\]
Finally, recall that by \eqref{apriori1}, $\Phi^\ep(y) = \text{O}(\ep)$. Hence, \eqref{hp8} holds. 
\vs
It then follows from \eqref{hp32} that, as $\ep\to 0$ and $y\to x$,  
\[
\dfrac{\Phi^\ep(y)}{\ep} \to \dfrac{\sum_{i=1}^d |a_i(x)|^\sigma b_i(x)}{\sum_{i=1}^d  |a_i(x)|^\sigma},
\]
and, hence, the claim. 
\end{proof}

\subsubsection{\bf {Strict minimum but not smooth}} 
In this subsection we consider the potential 
\beq\label{ex1delta}
V(y) = |y|^{1+\delta} \ \ \text{with} \ \  \delta \in (0,1).
\eeq
Let $\phi$ and $\Phi^\ep$ be as in the previous subsection, but with as above. Note that the existence of  $\Phi^\ep$ follows from the strict convexity of $V$.
\vs 

For each nonempty $E\subseteq\{1,\ldots, d\}$, set 
\[
V_E = \{p\in \R^d: p_i = 0 \text{ for all } i\in E \text{ and } p_i\ne 0 \text{ for all } i\notin E\},
\]
and set 
\[
V_\emptyset= \{p\in \R^d: p_i\ne 0 \text{ for all } 1\le i\le d\}. 
\]
\begin{lemma}
Fix $x\in \R^d$. If $D\phi(x)\in V_E$ for some nonempty $E\subseteq\{1,\ldots, d\}$, then 
 \begin{align}\label{deltasup.a}
\frac{1}{2}\min_{i\in E} \phi_{x_ix_i}(x)\le \liminf_{y\to x, \ep \to 0} \df{\Phi^{\ep}(y)}{\ep}  \le \limsup_{y\to x, \ep \to 0}  \df{\Phi^{\ep}(y)}{\ep} \le \frac{1}{2}\max_{i\in E} \phi_{x_ix_i}(x),
 \end{align}
and, if $D\phi(x)\in V_\emptyset$, then
 \[
 \lim_{y\to x, \ep \to 0}  \df{\Phi^{\ep}(y)}{\ep}   =  \frac{\sum_{i=1}^d|\phi_{x_i}(x)|^{\delta-1}\phi_{x_ix_i}(x)}{2\sum_{i=1}^d|\phi_{x_i}(x)|^{\delta -1}}. 
\]
\end{lemma}
\begin{proof}

Proceeding exactly as in the proof of Lemma \ref{takis11111}, we end up,  for $\ep$ small and $y$ near $x$, with the pair of inequalities displayed in \eqref{athens2.1} and the a priori estimate \eqref{apriori1}.
\vs 
Next fix $x$ with  $D\phi(x)\in V_E$ for some nonempty $E$. Then, for $y$ sufficiently close to $x$, there is some $c>0$ such that $|\phi_{x_i}(y)| > c$ when $i\notin E$. Moreover,   for any $\eta>0$,  $|\phi_{x_i}(y)| < \eta$ for all $i\in E$ provided again that $y$ sufficiently close to $x$.
\vs

Henceforth in this proof, $C_1,C_2,\ldots$ will denote positive constants that have no dependence on the choice of $\eta$. Moreover, the notation $\text{O}$  will be used only when the implicit constants have no dependence on the choice of $\eta$. 
\vs
For each $i=1,\ldots,d$, let  $a_i$ and $b_i$ be defined as in \eqref{hp10}, set 
\[
\begin{split}
\alpha_i^\ep(y) &= \Phi^\ep(y) - \phi(y+\sqrt{\ep}e_i) + \phi(y), \\
\beta_i^\ep(y) &= \phi(y - \sqrt{\ep}e_i) - \phi(y) - \Phi^\ep(y), 
\end{split}
\]
and 
\[ Q_i^\ep(y) = V'(\alpha_i^\ep(y)) - V'(\beta_i^\ep(y)),\]
and note that, in view of \eqref{chicago11112}, 
\beq\label{qieq}
\sum_{i=1}^d Q_i^\ep(y) = 0,
\eeq
and, by \eqref{apriori1},
\[
\alpha_i^\ep(y)=-\sqrt{\ep} \phi_{x_i}(y) + \text{O}(\ep) \ \ \text{and} \ \ \beta_i^\ep(y)=-\sqrt{\ep} \phi_{x_i}(y) + \text{O}(\e).
\]
It follows that, when $i\notin E$ and $y$ is close to $x$ and $\ep$ is close enough to $0$,  $\alpha_i^\ep(y)$ and $\beta_i^\ep(y)$ are both of order $\sqrt{\ep}$ on the same side of the origin.

\vs

Since $V''(z)$ is of order $|z|^{\delta-1}$ when $z$ is close to the origin, the observations above yield that, when $i\notin E$, as as $y\to x$ and $\ep \to 0$,
\begin{align}\label{qic}
|Q_i^\ep(y)| &= \text{O}(\ep^{(\delta-1)/2}) |\alpha_i^\ep(y) - \beta_i^\ep(y)|\le \text{O}(\ep^{(\delta+1)/2}), 
\end{align}
where in the second step, we used \eqref{apriori1} and Taylor's expansion to deduce that
\begin{align}\label{alphabeta}
\alpha_i^\ep(y) - \beta_i^\ep(y) = 2\Phi^\ep(y) - (\phi(y+\sqrt{\ep}e_i) + \phi(y-\sqrt{\ep}e_i)-2\phi(y)) = \text{O}(\ep).
\end{align}
Now, suppose that for all $i\in E$ and along a sequence $y\to x$ and $\ep \to 0$
\begin{align}\label{imposs}
\Phi^\ep(y) >\ep b_i(y)+ \ep \eta^{(1-\delta)/2}.
\end{align}
Then, along  this sequence, for any $i\in E$,
\begin{equation}\label{alphabeta2}
\begin{split}
\alpha_i^\ep(y) - \beta_i^\ep(y) &= 2\Phi^\ep(y) - (\phi(y+\sqrt{\ep}e_i) + \phi(y-\sqrt{\ep}e_i)-2\phi(y)) \\
&= 2\Phi^\ep(y) - 2\ep b_i(y) + \text{O}(\ep^{3/2}) \\
&> 2\ep \eta^{(1-\delta)/2} + \text{O}(\ep^{3/2}),
\end{split}
\end{equation}
and 
\[
|\alpha_i^\ep(y)| = |-\sqrt{\ep}\phi_{x_i}(y) + \text{O}(\ep)| \le C_1 \sqrt{\ep} \eta + \text{O}(\ep), 
\]
and
\[|\beta_i^\ep(y)|= |-\sqrt{\ep}\phi_{x_i}(y) + \text{O}(\ep)| \le C_1 \sqrt{\ep} \eta + \text{O}(\ep).\]

Then the  properties  of $V'$  imply that 
\beq\label{hp12}
\begin{split}
Q_i^\ep(y) = V'(\alpha_i^\ep(y)) - V'(\beta_i^\ep(y)) &> C_2 (\sqrt{\ep}\eta)^{\delta -1} (\alpha_i^\ep(y) - \beta_i^\ep(y)) \\
&\ge C_3(\sqrt{\ep}\eta)^{\delta -1}(\ep \eta^{(1-\delta)/2} + \text{O}(\ep^{3/2}))\\
&= C_3\ep^{(\delta+1)/2} \eta^{-(1-\delta)/2} + \text{O}(\ep^{(\delta+2)/2})\eta^{\delta-1}. 
\end{split}
\eeq
Combining  \eqref{qic} and \eqref{hp12}, we see that, for all $i\in E$ and any sequence $y\to x$ and $\ep \to 0$ satisfying \eqref{imposs}, we have
\begin{align*}
\sum_{i=1}^d Q_i^\ep(y) &= \sum_{i\in E} Q_i^\ep(y) + \sum_{i\notin E} Q_i^\ep(y)\\
&\ge C_3\ep^{(\delta+1)/2} \eta^{-(1-\delta)/2} + \text{O}(\ep^{(\delta+1)/2})(1+\eta^{\delta-1});
\end{align*}
Recall that the implicit constant in the $\text{O}$ notation in the last line has no dependence on the choice of $\eta$. Thus, if $\eta$ is chosen small enough, the estimate above contradicts \eqref{qieq} for sufficiently small $\ep$. 
\vs

Hence, for a small enough  $\eta$, there is no sequence satisfying \eqref{imposs}, and, thus, 
\[
\limsup_{y\to x, \ep\to 0} \frac{\Phi^\ep(y)}{\ep} \le \max_{i\in E} b_i(x) + \eta^{(1-\delta)/2}. 
\]
Since $\eta$ is arbitrary, the last estimate  proves the leftmost inequality in \eqref{deltasup.a}. The rightmost inequality may be proved by a similar argument, by just assuming the reverse inequality in \eqref{imposs} and arriving at a contradiction.
\vs
Next, take any $x$ such that $D\phi(x)\in V_\emptyset$. In the sequel all statements are supposed to hold as $y\to x$ and $\ep \to 0$, a fact which will not be repeated from line to line. 
\vs

Then arguing as above, it is easy to see that for any $i$, $\alpha_i^\ep(y)$ and $\beta_i^\ep(y)$ are of order $\sqrt{\ep}$ and, eventually, on the same side of the origin. 
\vs 

Therefore, by \eqref{qieq} and the smoothness of $V$ in $\R\setminus \{0\}$, we find 
\begin{align*}
0 &= \sum_{i=1}^d (\alpha_i^\ep(y) - \beta_i^\ep(y)) V''(\beta_i^\ep(y)) +\text{O}(\ep^2\ep^{(\delta-2)/2}),
\end{align*}
where the remainder term was obtained using the fact that $V'''(z) = \text{O}(|z|^{\delta-2})$ near zero, and \eqref{alphabeta}. 

\vs
Then, by \eqref{alphabeta2}, we get 
\begin{align*}
\Phi^\ep(y) \sum_{i=1}^d V''(\beta_i^\ep(y)) = \ep \sum_{i=1}^d b_i(y) V''(\beta_i^\ep(y)) + \text{O}(\ep^{3/2}) \sum_{i=1}^d V''(\beta_i^\ep(y)) + \text{O}(\ep^{(\delta+2)/2}),
\end{align*}
and consequently, since $V''$ is strictly positive everywhere,
\begin{align*}
\frac{\Phi^\ep(y)}{\ep} &= \frac{ \sum_{i=1}^d b_i(y) V''(\beta_i^\ep(y))}{\sum_{i=1}^d V''(\beta_i^\ep(y)) } +  \text{O}(\sqrt{\ep}) + \frac{\text{O}(\ep^{\delta/2})}{\sum_{i=1}^d V''(\beta_i^\ep(y))}.
\end{align*}
\vs
Since $\beta_i^\ep(y) = -\sqrt{\ep}\phi_{x_i}(y) + \text{O}(\ep)$ and $\phi_{x_i}(x)\ne0$ for all $i$, and $V''(z) = (1+\delta)\delta |z|^{\delta-1}$ for all $z\in \R$, we have that
\begin{align*}
\sum_{i=1}^d V''(\beta_i^\ep(y)) &= (1+\delta)\delta \sum_{i=1}^d |-\sqrt{\ep}\phi_{x_i}(y) + \text{O}(\ep)|^{\delta-1}\\
&= (1+\delta)\delta \ep^{(\delta-1)/2} \sum_{i=1}^d|\phi_{x_i}(y)|^{\delta-1} + \text{O}(\ep^{(\delta-1)/2}\ep),
\end{align*}
and, hence, 
\[
\frac{\text{O}(\ep^{\delta/2})}{\sum_{i=1}^d V''(\beta_i^\ep(y))} = \text{O}(\sqrt{\ep})
\]
Finally, note that 
\begin{align*}
\sum_{i=1}^d b_i(y) V''(\beta_i^\ep(y)) &= (1+\delta)\delta \sum_{i=1}^d (b_i(x)+\text{o}(1)) |-\sqrt{\ep}(\phi_{x_i}(x)+\text{o}(1)) + \text{O}(\ep)|^{\delta-1}\\
&= (1+\delta)\delta \ep^{(\delta-1)/2}\sum_{i=1}^d b_i(x)|\phi_{x_i}(x)|^{\delta-1} + \text{o}( \ep^{(\delta-1)/2}).
\end{align*}
The second claim of the lemma is now proved by combining the last four displays.
\end{proof}

Define  $F \in C(\mathcal S^d\times V_\emptyset)$ as  
\beq\label{Fdef}
F(X,p)= \frac{\sum_{i=1}^d |p_i|^{\delta-1} X_{ii}}{2\sum_{i=1}^d |p_i|^{\delta-1}}.
\eeq
The boundary of $V_\emptyset$ is the union of $V_E$ over all nonempty $E$. Take any nonempty $E$ and any $(X,p)\in V_E$. Take any sequence $(Y,q)$ in $\mathcal S^d \times V_\emptyset$ converging to $(X,p)$. Since $q_i\to 0$ for $i\in E$ and $q_i\to p_i\ne 0$ for $i\notin E$, the above formula for $F$ makes it clear that the $\limsup$ of $F(Y,q)$ as $(Y,q)\to (X,p)$ is at most $\frac{1}{2}\max_{i\in E} X_{ii}$. Moreover, we can choose a sequence $(Y,q)$ so that this value is achieved. In other words, for $(X,p)\in \mathcal S^d \times V_E$, 
\[
F^\star(X,p) = \frac{1}{2}\max_{i\in E} X_{ii},
\]
where $F^\star$ is defined as in \eqref{fstar1}. Similarly, 
\[
F_\star(X,p) =  \frac{1}{2}\min_{i\in E} X_{ii}. 
\]
For $u_0\in \text{BUC}(\R^d)$ and $F$ given by \eqref{Fdef} we consider the initial value problem 
\beq\label{takis1460a}
u_t=F(D^2u,Du) \ \  \text{in} \ \ \rt \  \ \text{and} \ \   u(\cdot,0)=u_0,
\eeq
which can be formulated as \eqref{general_ivp1} with 
\beq\label{hp20}
\overline F=\underline F=F \ \ \text{in} \ \ \mathcal S^d\times V_\emptyset \ \ \text{and} \ \ 
\overline F=F^\star \ \ \text{and} \ \ \underline F=F_\star \ \ \text{in} \ \ \mathcal S^d\times \partial V_\emptyset.
\eeq
It follows from It follows from Theorem~5 in \cite{MS} that the initial value problem  \eqref{general_ivp1} with $\overline F$ and $\underline F$ as in \eqref{hp20} satisfies \eqref{lh1}. 
\vs
For completeness, following \cite{MS}, we remark that the geometry associated with \eqref{takis1460a} and  $F$, $\overline F$ and $\underline F$  is rather  complicated. Indeed, the discontinuities  of $F$, $\overline F$ and $\underline F$ are characterized by a Finsler norm $\phi$ is defined  implicitly through its dual norm 
given, for each $p\in \R^d$,  by
\[\phi^\star (p) = \max \{ \underset{e' \in A\setminus \{e,-e\}} \sum |(p,e')| \; : e\in A \},\]
which defines a polyhedral Finsler norm. Above, $A$ is the set used at the beginning of Section~\ref{setup} to define the schemes.
\vs
We may now apply Theorem~\ref{thm1} to state our second main result.  
\begin{theorem}\label{thm100}
Assume  \eqref{ex1delta}.  Then, for every $u_0 \in \textup{BUC}(\R^d)$, the scheme defined by \eqref{ex2.2} converges, as $\ep \to 0$ and locally uniformly in $\R^d \times [0,T]$, to the unique solution of \eqref{general_ivp1} with $F, \overline F$ and $\underline F$  given by  \eqref{Fdef} and \eqref{hp20}.
\end{theorem}

\subsubsection{\bf {The solid-on-solid potential}} 
We consider here the  potential 
\beq\label{ex3.1.a}
V(y)= |y|  
\eeq
which is the potential for the solid-on-solid (SOS) model. This model has attracted considerable attention in the statistical physics literature. For a recent survey of rigorous results about the SOS model and various applications of the model, see the introduction of~\cite{H19}.
\vs
Next we introduce some more notation. Given $v_{\pm 1}, \ldots, v_{\pm d}$, we look at points minimizing the map
\[y\mapsto \sum_{i=1}^d |y-v_i| + |y-v_{-i}|,\] 
\vs
which, in view of the previous discussion about minima of  convex functions, form an interval $[a,b]$ with the understanding that it may be the case that $a=b$, with  $a=a(v_{\pm 1}, \ldots, v_{\pm d})$ and $b=b(v_{\pm 1}, \ldots, v_{\pm d})$. We define $\Phi(v_{\pm 1}, \ldots, v_{\pm d})$ as 
\beq\label{takis1420}
\Phi(v_{\pm 1}, \ldots, v_{\pm d})=\dfrac{ a(v_{\pm 1}, \ldots, v_{\pm d}) + b(v_{\pm 1}, \ldots, v_{\pm d})}{2}.
\eeq

\vs 
We recall that, given  real numbers $x_1,\ldots,x_n$,  its {median}  is any number $x$ that minimizes the map 
$$x\mapsto \sum_{i=1}^n |x_i-x|.$$

Thus, $\Phi(v_{\pm 1}, \ldots, v_{\pm d})$ is simply the midpoint of the set of medians of the numbers $v_{\pm 1}, \ldots, v_{\pm d}$. Lemma~\ref{takis1410} and Lemma~\ref{medianlemma.a}, which are stated and proved at the end of the ongoing subsection, summarize the properties of medians used in the proof of the convergence of the growth scheme generated by \eqref{ex3.1.a}. 
\vs

The scheme at scale one is defined by the rule
\[
u(x,t+1)=\varphi((u(x+a, t))_{a\in A})=u(x) + \Phi(u(x\pm e_1,t)-u(x,t), \dots,  u(x\pm e_d,t)-u(x,t)).
\]
It is immediate that $\varphi(0)=0$. Moreover, the scheme is clearly equivariant under translations by constants and monotone. 
\vs

At parabolic scale,  the scheme $S(\ep):\mathcal B(\R^d)\to \mathcal B(\R^d)$ is given by 
\[
S(\ep)v(x)=v(x) + \Phi(v(x\pm \sqrt \ep e_1)-v(x), \dots,  v(x\pm \sqrt \ep e_d)-v(x)).
\]

 To conclude we need to study, for $\phi:\R^d\to \R$ smooth,  the limit, as $\ep \to 0$ and $y\to x$ of 
 \[\df{S(\ep)\phi(y)-\phi(y)}{\ep}=\df {\Phi(\phi(y\pm \sqrt \ep e_1)-\phi(y), \dots,  \phi(y\pm \sqrt \ep e_d)-\phi(y))}{\ep}=\df{\Phi^{\ep}(y)}{\ep}.
\]

The following lemma gives the asymptotic behavior of $\Phi^{\ep,y}$ as $\ep \to 0$ and $y \to x$. For this it is convenient to introduce a partition of $\R^d$. For each nonempty $E\subseteq \{1,\ldots,d\}$, let 
\[
V_E = \{p\in\R^d: |p_i| = |p_j| \text{ for all } i,j\in E \text{ and } |p_i| < |p_j| \text{ for all } i\in E, j\notin E\}.
\]
When $E$ is a singleton set like $\{i\}$, we will write $V_i$ instead of $V_{\{i\}}$.

\vs
\begin{lemma}\label{medianlemma2.a}
Fix $x$ and let $\phi:\R^d\to \R$ be smooth. Let $\Phi^\ep$ be defined as in equation \eqref{Phidef}. If $D\phi(x)\in V_E$, then 
 \begin{align}\label{mediansup.a}
\frac{1}{2}\min_{i\in E} \phi_{x_ix_i}(x)\le \liminf_{y\to x, \ep \to 0} \df{\Phi^{\ep}(y)}{\ep}  \le \limsup_{y\to x, \ep \to 0}  \df{\Phi^{\ep}(y)}{\ep} \le \frac{1}{2}\max_{i\in E} \phi_{x_ix_i}(x).
 \end{align}
Moreover, if  $D\phi(x)\in V_i$ for some $i$, then
 \beq\label{medianconsistency.a}
 \lim_{y\to x, \ep \to 0}  \df{\Phi^{\ep}(y)}{\ep}   =  \frac{1}{2}\phi_{x_ix_i}(x). 
\eeq
\end{lemma}
\begin{proof}
Assume that $D\phi(x)\in V_E$ and suppose that $(y,\ep)$ is sufficiently close to $(x,0)$.
\vs

Since 
\begin{align}\label{medianphi.a}
\phi(y\pm \sqrt{\ep}e_i) - \phi(y) = \pm \sqrt{\ep} \phi_{x_i}(y) + \tO(\ep),
\end{align}
it is easy to see that the set of medians of $\phi(y\pm \sqrt{\ep}e_1) - \phi(y),\ldots, \phi(y\pm \sqrt{\ep}e_d) - \phi(y)$ 
is the same as the set of medians of $\left(\phi(y\pm \sqrt{\ep}e_i) - \phi(y)\right)_{i\in E}$.
\vs

Let
\[
a_i = \frac{1}{2}(\phi(y+ \sqrt{\ep}e_i)+ \phi(y- \sqrt{\ep}e_i) - 2\phi(y)) \ \text{and} \
b_i = \frac{1}{2}(\phi(y+ \sqrt{\ep}e_i)  - \phi(y- \sqrt{\ep}e_i)),
\]
so that 
\[
\phi(y\pm \sqrt{\ep}e_i) - \phi(y)  = a_i \pm b_i. 
\]
It follows from  Lemma \ref{medianlemma.a} (below) that 
\begin{align*}
\Phi^{\ep}(y)\le \max_{i\in E} a_i = \max_{i\in E} \biggl(\frac{\ep}{2} \phi_{x_ix_i}(y) + \too(\ep)\biggr)
= \frac{\ep}{2} \max_{i\in E} \phi_{x_ix_i}(y) + \too(\ep),
\end{align*}
which proves the upper bound in \eqref{mediansup.a}. The lower bound follows similarly. 
\vs
If  $E = \{i\}$ for some $i$, it is again easy to see from  \eqref{medianphi.a} that,  if  $y$ is sufficiently close to $x$ and $\ep$ is sufficiently close to $0$, 
\begin{align*}
\Phi^{\ep}(y) = \frac{1}{2}[(\phi(y+ \sqrt{\ep}e_i) - \phi(y)) + (\phi(y- \sqrt{\ep}e_i) - \phi(y))]
= \frac{\ep}{2}\phi_{x_ix_i}(y)  + \too(\ep), 
\end{align*}
which  proves the second claim. 
\end{proof}
\vs

Define  $F \in C(\mathcal S^d\times(\displaystyle{\bigcup}_{i=1}^d V_i))$ as  
\beq\label{medianF.a}
F(X,p)= \frac{ X_{ii}}{2} \ \ \text{if } \ \ p\in V_i.
\eeq

Then \eqref{medianconsistency.a} yields the consistency property if $D\phi (x) \in  \displaystyle{\bigcup}_{i=1}^d V_i$. 
\smallskip

To complete the argument, we need to analyze what happens on $\partial {\bigcup}_{i=1}^d V_i$. Note that $\partial {\bigcup}_{i=1}^d V_i$ can be written as a disjoint union of $V_E$'s over all $E$ of size at least $2$.
\vs 
Fix $(X,p)\in \mathcal{S}^d\times V_E$ and suppose that $(Y,q)\in \mathcal{S}^d\times (\bigcup_{i=1}^d V_i)$ converges to $(X,p)$. Passing to a subsequence, we may assume that there exists some $i$ such that $q\in V_i$. Thus, $F(Y,q)=\frac{1}{2}Y_{ii}$, and, hence, $F(Y,q)$ converges to $\frac{1}{2}X_{ii}$. This implies, in particular, that any such $i$ must be a member of $E$. 
\vs
Thus, 
$$F^\star(X,p)\le \frac{1}{2} \max_{i\in E} X_{ii} \ \text{ and } \  F_\star(X,p)\ge \frac{1}{2} \min_{i\in E} X_{ii}.$$
\vs 

If $i^*\in E$ is a coordinate at which $X_{ii}$ is maximized, then choosing $(Y,q)\to (X,p)$ so that $q\in V_{i^*}$ always, we get $F^\star(X,p)\ge \frac{1}{2}\max_{i\in E} X_{ii}$, and,  similarly, $F_\star(X,p)\le \frac{1}{2}\min_{i\in E} X_{ii}$. 
\vs

Thus, we conclude that, when $p\in V_E$,
\[
F^\star(X,p)= \frac{1}{2} \max_{i\in E} X_{ii} \ \text{and} \  F_\star(X,p) = \frac{1}{2}\min_{i\in E}X_{ii}.
\]
It is now easy to verify \eqref{consistency} using Lemma \ref{medianlemma2.a} and taking $M = S^{d-1} \cap (\bigcup_{|E|\ge 2} V_E)$. 
\vs
For $u_0\in \text{BUC}(\R^d)$ and $F$ given by \eqref{medianF.a} we consider the initial value problem 
\[
u_t=F(D^2u,Du) \ \  \text{in} \ \ \rt \  \ \text{and} \ \   u(\cdot,0)=u_0,
\]
which can be formulated as  \eqref{general_ivp1} with 
\beq\label{hp21}
\overline F=\underline F=F \ \ \text{in} \ \ \mathcal S^d\times V_\emptyset \ \ \text{and} \ \ 
\overline F=F^\star \ \ \text{and} \ \ \underline F=F_\star \ \ \text{in} \ \ \mathcal S^d\times \partial V_\emptyset.
\eeq
It follows from It follows from Theorem~5 in \cite{MS} that the initial value problem  \eqref{general_ivp1} with $\overline F$ and $\underline F$ as in \eqref{hp21} satisfies \eqref{lh1}. 
\vs
As explained in \cite{MS}, the comparison principle for  the initial value problem with nonlinearity \eqref{medianF.a} is based on the fact that $F$  is  compatible with a particular  polyhedral Finsler norm whose dual 
is given, for $p\in \R^d$ and $A$ as in the beginning of Section~\ref{setup}, by  
\[\underline\phi^\star(p)=\max \{ \underset {e' \in A\setminus \{e,-e\}} \sum  |(p,  e')| \;: e\in A \}.\]
\vs
Combining all the above we have now the next result of the paper.

\begin{theorem}\label{thm5}
Assume \eqref{ex3.1.a}. Then the scheme defined using \eqref{takis1420} converges, as $\ep \to 0$ and locally uniformly, to the unique solution of \eqref{ivp1} with $F$ as in \eqref{medianF.a}.
\end{theorem}

We conclude with the properties of the median of points $x_1,\ldots,x_n\in \R^d$.
\begin{lemma}\label{takis1410}
Let $x_1,\ldots,x_n$ be real numbers. Then:\\
(i)~The set of medians is always a closed interval (which may be a single point).\\
(ii)~A point $x$ is a median if and only if $|\{i: x_i \ge x\}| \ge n/2$ and $|\{i: x_i \le x\}|\ge n/2$. \\
(iii)~If the set of medians is an interval $[a,b]$ with $a<b$, then $n$ must be even, and\\
 $|\{i: x_i \le a\}| = |\{i: x_i \ge b\}| = n/2$, and no $x_i$ is in $(a,b)$. 
\end{lemma}
{\begin{proof} The arguments are easy and well known so we only present a sketch.
\vs
The first property is a simple consequence of the convexity of the map $x\mapsto \sum_i |x_i-x|$. 
\vs 
To prove the second property, take any $x$, and let $k$, $l$ and $m$ be respectively the number of $i$ such that $x_i$ is greater than, equal to, and less than $x$. Increasing $x$ to $x+\ep$, for small enough $\ep>0$, increases $\sum_i |x_i -x|$ by $(l+m - k)\ep$. Thus, if $x$ is a median, then $l+m-k$ must be nonnegative, which is the same as saying that $ |\{i: x_i \le x\}| \ge n/2$. Similarly, it can be shown that decreasing $x$ to $x-\ep$ increases the sum by $k+l-m$, which implies that $k+l-m\ge 0$, which is the same as saying $ |\{i: x_i \ge x\}| \ge n/2$. Conversely, the same argument shows that if these two inequalities hold, then shifting $x$ a little to the left or the right cannot decrease $\sum_i |x_i-x|$, and so, $x$ is a median. 
\vs
Finally, to prove the third property, suppose that the set of medians is an interval $[a,b]$ with $a< b$. Take any $x\in (a,b)$. Let $k$, $l$ and $m$ be as above. Since $[a,b]$ is the set of medians, the above argument shows that $k+l-m$ and $l+m-k$ must both be equal to zero, which implies that $l=0$. Thus, no $x_i$ can belong to the interval $(a,b)$. Thus, for any $x\in (a,b)$, the sets $\{i:x_i \le x\}$ and $\{i: x_i \ge x\}$ are disjoint. Since $x$ is a median, (ii)  now implies that both of these sets must have size exactly $n/2$.
\end{proof}

The next lemma yields yet another property of the median which is relevant for the problem at hand. 
\begin{lemma}\label{medianlemma.a}
Let $a_1,\ldots,a_d, b_1,\ldots, b_d\in \R$.  Then $\max_i a_i$ and $\min_i a_i$ are respectively an upper and a lower bound for the midpoint of the set of medians of the $2d$ numbers $a_1\pm b_1,\ldots, a_d\pm b_d$. 
\end{lemma}
\begin{proof}
We prove  the upper bound. Without loss of generality, we may assume that $a_1\ge a_i$ for all $i$, and $b_1,\ldots, b_d$ are nonnegative. For $i=1,\ldots,d$, let $c_{2i-1} = a_i+b_i$ and $c_{2i} = a_i - b_i$. 
\vs

If the median is a unique point $x$, then any $y<x$ is not a median, and, hence, $|\{i: c_i \le y\}|<d$ because $|\{i: c_i \ge y\}|\ge |\{i: c_i \ge x\}| \ge d$. Moreover, for any $i$, we have $a_i -b_i \le a_i \le a_1$. Thus, $|\{i: c_i\le a_1\}|\ge d$, and, therefore, $a_1 \ge x$. 
\vs

Next, suppose that the set of medians is an interval $[a,b]$ and let $x = (a+b)/2$ be the midpoint of this interval. If  $a_1 < a$, then,  since $a$ is a median, $|\{i: c_i \ge a_1\}| \ge |\{i: c_i \ge a\}|\ge d$. It also  follows from Lemma~\ref{takis1410} that $|\{i: c_i\le a_1\}|\ge d$. Thus, $a_1$ must be a median, which is a contradiction since $a_1 \notin [a,b]$.
Hence, $a_1\ge a$. 
\vs

Assume that $a_1\in [a, x)$.  Then $a_i < x$ for all $i$. We know, however,  that exactly $d$ of the $c_i$'s are $\ge b$, and this can happen only if $a_i + b_i\ge b$ for each $i$. Since $a_i <x$, this implies that $a_i-b_i < a$ for each $i$, and,  in particular,   no $c_i$ can be equal to $a$. But then $a-\ep$ is also a median for sufficiently small $\ep$ which is a contradiction. Hence, $a_1\ge x$, and the proof of the upper bound is complete.
\vs
For the lower bound, it is enough to work with $-a_1,\ldots,-a_d, -b_1,\ldots,-b_d$ and apply the upper bound. 
\end{proof}

\subsubsection{\bf The non-strict minimum case.}  
We continue now with the case that $0$ is not a strict  minimum of $V$, in which case,   the symmetry of $V$ implies  that there exists $a>0$ such that
\beq\label{ex2.6.a}
V=0 \ \ \text{in} \ \ [-a,a] \ \ \text{and} \ \ V>0 \ \ \text{in} \ \ \R\setminus [-a,a].
\eeq 
This kind of potential arises in the so-called restricted solid-on-solid (RSOS) models, introduced by Kim and Kosterlitz~\cite{KK89}. The general principle of RSOS models is that the heights at neighboring points are restricted to be within some constant of each other. Very little is known rigorously about these models (see \cite{C2} for some recent results). In this subsection we will study a deterministic version of RSOS growth, induced by the potential displayed above. 
\vs
Fix $v_{\pm 1},\ldots, v_{\pm d} \in [-a, a]$. It follows that, for all $y \in [\underset{i=1,\ldots,d}\max v_{\pm i} - a, \underset{i=1,\ldots,d}\min v_{\pm i} +a]$,
\[\sum_{i=1}^d V(y-v_{-i}) + V(y-v_i) =0,\]
that is, the map $y \to \sum_{i=1}^d V(y-v_{-i}) + V(y-v_i)$ achieves a minimum, which is $0$, on the interval 
$[\underset{i=1,\ldots,d}\max v_{\pm i} - a, \underset{i=1,\ldots,d}\min v_{\pm i} +a].$
\vs

Following the discussion at the beginning of the ongoing section, we choose  the middle point of this interval. Thus the scheme we are working with here at scale one is defined, for $u:\Z^d\times \Z_+ \to \R$, by 
\beq\label{takis1400}
u(x,t+1)=\varphi ((u(x+a,t))_{a\in A})= u(x,t) + \df{1}{2}[\underset{b\in B}\max \ (u(x+b,t)-u(x,t)) + \underset{b\in B}\min \ (u(x+b,t)-u(x,t))]
\eeq
(Recall that $B = \{\pm e_1,\ldots,\pm e_d\}$ and $A = B\cup\{0\}$.) That  $\varphi(0)=0$,  and the  equivariance under translations by constants and monotonicity are immediate. 
\vs

It is immediate from the discussion about the choice of scale that, in the setting discussed here,  the ``correct scaling'' is the parabolic one.
\vs

At the parabolic  scale, the scheme is generated by the map $\ep \to S(\ep):\mathcal B \to \mathcal B$, given, for $v \in \mathcal B$, by 
\beq\label{1401}
S(\ep)v(x)=v(x) + \df{1}{2}\left[\underset{i=1,\ldots,d}\min [\phi (x \pm \sqrt \ep e_i)-\phi (x)]+  \underset{i=1,\ldots,d}\max [\phi (x \pm \sqrt \ep e_i)-\phi (x)]\right].
\eeq
  
In what follows, to simplify the notation, we  use the map $\Phi:\R^{2d}\to \R$ give by 
\[
\Phi (v_{-1}, v_1, \ldots, v_{-d}, v_d) = \frac{1}{2}\left[\underset{i=1,\ldots,d}\min v_{\pm i} +  \underset{i=1,\ldots,d}\max v_{\pm i}\right].
\]

Then \eqref{1401} can be rewritten as
\[
\begin{split}
&\Phi(\phi(x \pm \sqrt{\ep}e_1)-\phi (x), \ldots,\phi(x \pm \sqrt{\ep}e_d)-\phi (x))\\[2mm]
&\qquad \qquad = \frac{1}{2} \left[\underset{i=1,\ldots,d}\min [\phi (x \pm \sqrt \ep e_i)-\phi (x)]+  \underset{i=1,\ldots,d}\max [\phi (x \pm \sqrt \ep e_i)-\phi (x)]\right].
\end{split}
\]

In view of the previous observations, the only fact we need to check is the consistency of the scheme, which is about the behavior, as  $(\ep,y) \to (0,x)$ and $\phi$ smooth, of the ratio 
\[\df{S(\ep)(\phi)(y)-\phi(y)}{\ep}=\df{\Phi^{\ep}(y)}{\ep}, \]
where, to ease the notation, we write 
\begin{align}\label{maxphi.a}
\Phi^{\ep}(y) =\dfrac{1}{2}\biggl[\min_{i=1,\ldots,d} [\phi(y \pm \sqrt \ep e_i)-\phi (y)]+ \max_{i=1,\ldots,d} [\phi(y \pm \sqrt \ep e_i)-\phi(y)]\biggr].
\end{align}

\vs
The following lemma identifies the asymptotic behavior of $\ep^{-1}\Phi^{\ep}(y)$ as $\ep \to 0$ and $y\to x$ for some $x$.
\vs

For the statement, we  introduce, for each nonempty $E\subseteq \{1,\ldots,d\}$, the subset $V_E$ of $\R^d$ given by   
\[
V_E = \{p\in\R^d: |p_i| = |p_j| \text{ for all } i,j\in E \text{ and } |p_i| > |p_j| \text{ for all } i\in E, j\notin E\};
\]
when $E$ is a singleton set like $\{i\}$, we will write $V_i$ instead of $V_{\{i\}}$.  Although this bears similarities with the one of the previous subsection, is just the ``opposite''. Nevertheless, to keep the notation under control we use the same symbols.
\vs

\vs

It is immediate that the $V_E$'s form a partition of $\R^d$.
\vs
\begin{lemma}\label{maxphilemma.a}
Let $\phi:\R^d \to \R$ be smooth and $x\in \R^d$. If $D\phi(x)\in V_E$,  then, locally uniformly in $x$,
\[
 \dfrac{1}{2}\min_{j\in E} \phi_{x_jx_j}(x) \le \liminf_{y\to x, \ep \to 0} \df{\Phi^{\ep}(y)}{\ep}\le \limsup_{y\to x, \ep \to 0} \df{\Phi^{\ep}(y)}{\ep} \le \frac{1}{2}\max_{j\in E} \phi_{x_jx_j}(x). 
\]
Moreover, if $D\phi(x)\in V_i$ for some $i$, then
\beq\label{ex2.14.a}
\lim_{y\to x, \ep \to 0} \df{\Phi^{\ep}(y)}{\ep}= \frac{1}{2} \phi_{x_ix_i}(x).
\eeq
\end{lemma}
\begin{proof}
For $y$ near  $x$ and $\ep$ small, we have
\begin{align*}
\Phi^{\ep}(y) &= \frac{1}{2}\biggl[\min_{i=1,\ldots,d} \biggl[\pm\sqrt{\ep} \phi_{x_i}(y) + \frac{\ep}{2}\phi_{x_ix_i}(y) + \too(\ep)\biggr]+ \max_{i=1,\ldots,d} \biggl[\pm\sqrt{\ep} \phi_{x_i}(y) + \frac{\ep}{2}\phi_{x_ix_i}(y) + \too(\ep)\biggr]\biggr].
\end{align*}

First, suppose that $D\phi(x)\in V_i$ for some $i$, which, by the definition of $V_i$,  implies  that $\phi_{x_i}(x)>0$. Then the above expression shows that, when $y$ and $\ep$ are respectively sufficiently close to $x$ and $0$,
\begin{align*}
\Phi^{\ep}(y) &= \frac{1}{2}\biggl[-\sqrt{\ep} \phi_{x_i}(y) + \frac{\ep}{2}\phi_{x_ix_i}(y) + \too(\ep)\biggr]+ \frac{1}{2}\biggl[\sqrt{\ep} \phi_{x_i}(y) + \frac{\ep}{2}\phi_{x_ix_i}(y) + \too(\ep)\biggr]\\
&= \frac{\ep}{2}\phi_{x_ix_i}(y) + \too(\ep),
\end{align*}
which proves the second claim of the lemma. 
\vs

Next, we assume  that $D\phi(x)\in V_E$ for some $E$. When $y$ and $\ep$ are sufficiently close to $x$ and $0$ respectively, the maximum in \eqref{maxphi.a} is attained at $e_i$ or $-e_i$ for some $i\in A$, which possibly depends on $y$ and $\ep$. 
\vs

Suppose that the maximum is attained at  $e_i$ for some $i\in E$. Then
\begin{align*}
\Phi^{\ep}(y) &= \frac{1}{2} \min_{j=1,\ldots,d} \biggl[\pm\sqrt{\ep}\phi_{x_j}(y) + \frac{\ep}{2} \phi_{x_jx_j}(y) + \too(\ep)\biggr] + \frac{1}{2} \biggl[\sqrt{\ep}\phi_{x_i}(y) + \frac{\ep}{2} \phi_{x_ix_i}(y) + \too(\ep)\biggr] \\
 &\le \frac{1}{2} \biggl[-\sqrt{\ep}\phi_{x_i}(y) + \frac{\ep}{2} \phi_{x_ix_i}(y) + \too(\ep)\biggr] + \frac{1}{2} \biggl[\sqrt{\ep}\phi_{x_i}(y) + \frac{\ep}{2} \phi_{x_ix_i}(y) + \too(\ep)\biggr] \\
&= \frac{\ep}{2} \phi_{x_ix_i}(y) + \too(\ep)\le \frac{\ep}{2}\max_{j\in E} \phi_{x_jx_j}(y) + \too(\ep).
\end{align*}
\vs

Similarly, if the maximum is attained at $-e_i$ for some $i\in E$, then 
\begin{align*}
\Phi^{\ep}(y) &= \frac{1}{2} \min_{j=1,\ldots,d} \biggl[\pm\sqrt{\ep}\phi_{x_j}(y) + \frac{\ep}{2} \phi_{x_jx_j}(y) + \too(\ep)\biggr] + \frac{1}{2} \biggl[-\sqrt{\ep}\phi_{x_i}(y) + \frac{\ep}{2} \phi_{x_ix_i}(y) + \too(\ep)\biggr] \\&\le \frac{1}{2} \biggl[\sqrt{\ep}\phi_{x_i}(y) + \frac{\ep}{2} \phi_{x_ix_i}(y) + \too(\ep)\biggr] + \frac{1}{2} \biggl[-\sqrt{\ep}\phi_{x_i}(y) + \frac{\ep}{2} \phi_{x_ix_i}(y) + \too(\ep)\biggr] \\
&= \frac{\ep}{2} \phi_{x_ix_i}(y) + \too(\ep)\le \frac{\ep}{2}\max_{j\in E} \phi_{x_jx_j}(y) + \too(\ep).
\end{align*}
A similar argument yields that 
\begin{align*}
\Phi^{\ep}(y) &\ge \frac{\ep}{2}\min_{j\in E} \phi_{x_jx_j}(y) + \too(\ep).
\end{align*}
The proof is now  complete. 
\end{proof}
\vs
We introduce next the limiting pde. Let  $F \in C(\mathcal S^d\times \overset{d}{\underset{i=1}{\bigcup}} V_i )$ given by 
\beq\label{ex2.15.a}
F(X,p)= \frac{X_{ii}}{2} \ \ \text{if } \ \ p\in V_i.
\eeq

Then \eqref{ex2.14.a} yields the consistency property if $D\phi (x) \in  \overset{d}{\underset{i=1}{\bigcup}} V_i$. 
\vs

To complete the definition of $F$  we need to analyze what happens on $\partial (\overset{d}{\underset{i=1}{\bigcup}} V_i)$,  can be written as a disjoint union of $V_E$ over all $E$ of size at least $2$.
\vs 

Fix $(X,p)\in \mathcal{S}^d\times V_E$ and consider  a sequence $(Y,q)\in \mathcal{S}^d\times \overset{d}{\underset{i=1}{\bigcup}} V_i$ converging to $(X,p)$. Passing to a subsequence, we may assume that there exists some $i$ such that $q\in V_i$ always. Thus, $F(Y,q)=\frac{1}{2}Y_{ii}$, and, hence, $F(Y,q)$ converges to $\frac{1}{2}X_{ii}$. This implies, in particular, that any such $i$ must be a member of $E$. It follows that
$$F^\star(X,p)\le \frac{1}{2}\max_{i\in E} X_{ii} \ \text{and} \  F_\star(X,p)\ge \frac{1}{2} \min_{i\in E} X_{ii}.$$ 
\vs 

If, however,  $i^*\in E$ is a coordinate at which $X_{ii}$ is maximized, then it is easy to see, and we leave the details to the reader, that we can  choose $(Y,q)\to (X,p)$ so that $q\in V_{i^*}$ always. 
It follows that  $$F^\star(X,p)\ge \frac{1}{2}\max_{i\in E} X_{ii},$$
and, similarly, $$F_\star(X,p)\le \frac{1}{2}\min_{i\in E} X_{ii}.$$ 
\vs

Thus, we conclude that, when $p\in V_E$,
\beq\label{2.19.a}
F^\star(X,p)= \frac{1}{2} \max_{i\in E} X_{ii} \ \ \text{and}  \ \ F_\star(X,p) = \frac{1}{2}\min_{i\in E}X_{ii}.
\eeq
It is now easy to verify \eqref{consistency} using Lemma \ref{maxphilemma.a}.
\vs
For $u_0\in \text{BUC}(\R^d)$ and $F$ given by \eqref{ex2.15.a} we consider the initial value problem \eqref{ivp1} with $F^\star$ and $F_\star$ defined as in \eqref{2.19.a}, which can be formulated as  \eqref{general_ivp1} with 
\[
\overline F=\underline F=F \ \ \text{in} \ \ \mathcal S^d\times \overset{d}{\underset{i=1}{\bigcup}} V_i \ \ \text{and} \ \ 
\overline F=F^\star \ \ \text{and} \ \ \underline F=F_\star \ \ \text{in} \ \ \mathcal S^d\times \partial \overset{d}{\underset{i=1}{\bigcup}} V_i.
\]
It follows from It follows from Theorem~5 in \cite{MS} that the initial value problem  \eqref{general_ivp1} with $\overline F$ and $\underline F$ as in \eqref{hp21} satisfies \eqref{lh1}. 
\vs
It is  shown in \cite{MS} (see Lemma 3) that the nonlinearity $F$  defined above is encodes an infinity Laplacian and is compatible with 
the polyhedral norm $\phi_A$, where again $A$ is as in the beginning of Section~\ref{setup}, with dual  
\[\phi_A^\star(p)=\max \{(p,e) \; : e\in A\}.\]
Hence, in view of Theorem 5 in \cite{MS}, the corresponding initial value problem admits a comparison principle. 
\vs 
Combining all the above we have now the next result of the paper.

\begin{theorem}\label{thm4}
Assume \eqref{v.a} and \eqref{ex2.6.a}. Then the scheme defined using \eqref{1401} converges, as $\ep \to 0$ and locally uniformly in $\R^d\times [0,T]$, to the unique  solution of \eqref{ivp1} with $F$ as in \eqref{ex2.15.a}.
\end{theorem}

\bigskip

\subsection*{Acknowlegments}
The first author was partially supported by National Science Foundation grant DMS-1855484. The second author was partially supported by the National Science Foundation grant  DMS-1900599, the Office for Naval Research grant N000141712095 and the Air Force Office for Scientific Research grant FA9550-18-1-0494. 

The authors would also like to thank Ery Arias-Castro, Peter Morfe and Lexing Ying for helpful suggestions.

\end{document}